\documentclass[11pt]{article} 

\usepackage[margin=1in]{geometry}
\usepackage{subcaption}

\usepackage{amsmath}
\usepackage{amssymb}
\usepackage{amsthm}
\usepackage{bm}

\usepackage{graphicx}
\usepackage{booktabs}
\usepackage{caption}
\usepackage{float}
\usepackage{rotating}
\usepackage{xcolor}
\usepackage{tikz}

\usepackage{enumitem}
\usepackage{microtype}
\usepackage{titlesec}
\usepackage{cite}
\usepackage{blindtext}
\usepackage{lipsum}

\emergencystretch=\maxdimen
\usepackage[
    colorlinks=true,
    linkcolor=blue!60!black,
    citecolor=blue!60!black,
    urlcolor=blue!60!black
]{hyperref}

\titleformat{\section}{\large\bfseries}{\thesection}{0.6em}{}
\titleformat{\subsection}{\normalsize\bfseries}{\thesubsection}{0.6em}{}

\newcommand{\R}{\mathbb{R}}
\newcommand{\Sph}{\mathbb{S}^{2}}
\newcommand{\dd}{\,\mathrm{d}}
\newcommand{\E}{\mathbb{E}}
\newcommand{\Fb}{\frac{F}{m}}
\newcommand{\ph}{\varphi}

\newcommand{\grad}{\nabla}
\newcommand{\vx}{\mathbf{x}}
\newcommand{\vv}{\mathbf{v}}
\newcommand{\vz}{\mathbf{z}}

\newcommand{\relL}{\mathrm{Rel}\,L^2}

\theoremstyle{plain}
\newtheorem{lemma}{Lemma}

\title{\textbf{Weak Adversarial Neural Pushforward Method \\ for Boltzmann Equation}}
\author{Jenia Fardousi Koly, Andrew Qing He, Wei Cai\thanks{Corresponding author, cai@smu.edu. 8/7/2026.} \\ \\ Department of Mathematics, Southern Methodist University, Dallas, Texas, USA}
\date{}

\begin{document}
\maketitle

\begin{abstract}
    In this paper, we extend a weak adversary neural network pushforward method for solving time dependent Boltzmann equation and a weak formulation of the collision operator is proposed where an invertible neural pushforward mapping is used to generating samples given by the distribution governed by the Boltzmann equation. The training of the pushforward mapping is learnt by enforcing the weak form of the Boltzmann equation. Numerical results have demonstrated the effectiveness of the proposed method.  
\end{abstract}

\section{Introduction}

The Boltzmann equation (BE) provides a dynamic statistical description, through a density distribution in the phase space for classical systems of many particles, where applying Newtonian mechanics to individual particles is not feasible due to the large number of particles involved and the lack of knowledge of their initial states. As the phase space is six dimensional in addition to the time variable, the BE poses a challenge for computational simulations.

For solving the BE numerically, the large dimensional space of independent variables, namely the spatial variables and the velocity variables, needs to be considered. This can be achieved by first discretizing the distribution function in the phase space and then, in the spatial variables, as part of the discretization of the spatial derivatives. The main difficulty in solving numerically the BE comes from the collision term, since a straightforward quadrature of the latter would require $O(N^6)$ computations per spatial cell, where N is the number of grid points in the phase space \cite{jaiswal2019discontinuous}.Consequently, several approaches, such as Fourier–Galerkin spectral methods \cite{pareschi2000numerical}, have been put forth to deal with the dimensional cost of the numerical solution of the BE. Here, we concentrate on the deterministic technique to solving the BE. Fourier-Galerkin spectral methods have been developed to achieve spectral accuracy  in the calculation of the collision operator, while at the same time reducing the dimensional cost of the numerical solution, and also to allow for a large flexibility in the choice of the collision kernel. Fast spectral algorithms for solving the BE have been recently developed, see for example \cite{mouhot2006fast,gamba2017fast}. In addition to the spectral reduction of the computational cost in the phase space, for achieving a high order of accuracy in the physical space, in \cite{jaiswal2019discontinuous}, a Discontinuous Galerkin (DG) method has been recently proposed to solve the full BE. Related efforts include asymptotic-preserving schemes that remain accurate across collision regimes \cite{jin1999efficient} a solution representation that dynamically evolves in a low-dimensional space to alleviate the effects of dimensionality
\cite{einkemmer2018low,hu2022adaptive,wang2021efficient}. A comprehensive account of these: a survey of the abovementioned deterministic approaches for the numerical solution of the BE has also been recently completed by the authors Dimarco and Pareschi in \cite{dimarco2014numerical}. On the stochastic side, however, many attempts have been made to "hybridize" deterministic methods for the solution of the BE with a particle representation of the distribution function, direct simulation Monte Carlo (DSMC) method \cite{bird1994molecular,nanbu1980direct}, is well known to asymptotically converge to the solution of the BE in the limit of a very large number of particles \cite{wagner1992convergence}, much larger than the number of spatial grid points, and thus not competitive with a fully deterministic strategy in terms of the average with statistical error of order $1/\sqrt{N}$.

A recent development is the introduction of deep learning techniques using neural networks to solve PDEs from mathematical physics, including the BE \cite{raissi2019physics,yu2018deep,sirignano2018dgm,zang2020weak}. The deep learning approach reformulates the PDE solution as an optimization problem based on an objective loss function, which is usually based on some mathematical or physical property of the PDE. One of the main advantages of using a neural network is its ability to approximate high-dimensional functions; therefore, it is quite natural to apply it to find the solution of the BE. Several attempts in this direction have been carried out \cite{han2019uniformly,schotthofer2022structure,xiao2021using}.

In this paper, we will extend our recently developed weak adversarial neural pushforward method (WANPM) \cite{he2025learning,he2026weakFP,he2026weakMFP} to study the BE.
The WANPM is an indirect way to find the solution of the BE by generating samples for the solution---viewed as a probability density function (PDF) in the phase space---through deep learning, where the loss function is based on the weak form of the BE. The key idea is that, by using the weak form of the BE through integration by parts, the weak form turns into an expectation of the differentiated test function under the sought after BE governed PDF. Then, a  neural pushforward mapping between a simple base distribution in the base space and the desired BE phase space can then be learnt to give the desired PDF by minimizing the weak-form residual for the BE. For the BE, a special treatment of the weak form of the collision operator is needed, obtained by introducing a marginal PDF in the physical space, which extends the previous WANPM to the BE. Since the weak collision contribution is expressed as an expectation over samples generated by the learned pushforward and is thus estimated by Monte Carlo averaging over colliding pairs and scattering directions, its cost scales with the number of samples rather than with a tensor-product grid in velocity space. This treatment of the collision term in the BE will be shown to be advantageous over other methods in terms of the cost of evaluating the high-dimensional collision operator—a major challenge for traditional numerical methods.

WANPM offers several unique advantages over other  alternatives, deterministic mesh or collocation solvers such as physical informed neural networks (PINNs) and stochastic particle methods (DSMC). First, because the solution is represented as a \emph{pushforward} of a simple base law, $(\vx,\vv)=G_\theta(\vz,t)$, the model density is nonnegative by construction; this structural property residual-form PINNs must enforce specifically with difficulties. Second, the pushforward mapping approach automatically generates samples in regions as demanded by the BE solution, achieving an automatic adaptive representation of the solution. Mesh based numerical methods or PINNS face challenges in adaptivity in the high dimensional phase space. Third, the pushforward learns the \emph{joint} phase-space law rather than a pointwise density field, so it captures correlations between position and velocity that a density-fitting PINN systematically misses: on a free-transport BE, the PINN reproduces every one-dimensional marginal, yet returns a position-velocity covariance of essentially zero (relative error $99.96\%$) against the exact value of one, whereas WANPM recovers it to about $1.2\%$ (Section~\ref{sec:results}). Fourth, the weak adversarial residual enforces the equation in distributional form using trainable test functions. This helps detect errors that ordinary collocation residuals may miss, such as an incorrect joint density even when the marginals look correct. Fifth, the collision term is computed by Monte Carlo sampling of particle pairs and scattering directions, so the cost scales with sample size rather than a high-dimensional velocity grid. On the other hand, in comparison with DSMC, WANPM provides a single differentiable, mesh-free sampler that works across free transport, forced transport, and collisional regimes.

The rest of the paper is organized as follows. Section~\ref{sec:bte-forced} introduces the forced BE, the collision operator, and the elastic-collision involution. Section~\ref{sec:weakform} derives the weak form under an external force and the density-consistent treatment of the collision term. Section~\ref{sec:wanpm-forced} presents the WANPM pushforward construction, the adversarial objective, and the Monte Carlo estimators. Sections~\ref{sec:numerical_results}  report the free-transport benchmark E1 against DSMC and PINN, and Section~\ref{sec:E2} reports the harmonic-force benchmark E2. Section~\ref{sec:conclusion} concludes. Detailed exact-solution derivations and the DSMC and error-metric protocols are collected in the appendix.

\section{The Boltzmann equation under external force}
\label{sec:bte-forced}
The density distribution in the phase space for a classic particle system is denoted by
\[
f:\R_x^3\times \R_v^3\times [0,T]\to \R_{+},
\qquad f=f(x,v,t)
\]
The dynamics of the distribution  in the phase space is governed by the following Boltzmann transport equation with an external force $F$ acting on particles of mass $m$ in the non-conservative form
\begin{equation}
\partial_t f+v\cdot\nabla_x f+\frac{F}{m}\cdot\nabla_v f=Q[f,f].
\label{eq:forced_bte_nc}
\end{equation}
A corresponding conservative form is given by
\begin{equation}
\partial_t f+v\cdot \nabla_x f+
\nabla_v\cdot\left(\frac{F}{m}f\right)=Q[f,f].
\label{eq:forced_bte_c}
\end{equation}
By the product rule,
\[
\nabla_v\cdot\left(\frac{F}{m}f\right)
=
\frac{F}{m}\cdot\nabla_v f
+
f\,\nabla_v\cdot\left(\frac{F}{m}\right).
\]
Therefore, the two forms coincide when
\begin{equation}
\nabla_v\cdot\left(\frac{F}{m}\right)=0.
\label{eq:force_divergence_free}
\end{equation}
This holds, for example, when $F=F(x,t)$ is independent of $v$. It also holds for velocity-divergence-free forces such as the Lorentz force $F=q(E+v\times B_{\rm mag})$.

In \eqref{eq:forced_bte_nc}, the Boltzmann collision operator is given by
\begin{equation}
\begin{aligned}
Q[f,f](x,v,t)
&=
\int_{\R^3}\int_{\Sph}
B(v-v_*,\omega)
\Big[f(x,v',t)f(x,v_*',t)-f(x,v,t)f(x,v_*,t)\Big]
\dd\omega\dd v_* ,
\end{aligned}
\label{eq:collision_operator}
\end{equation}
where
\begin{equation}
v'=\frac{v+v_*}{2}+\frac{|v-v_*|}{2}\omega,
\qquad
v_*'=\frac{v+v_*}{2}-\frac{|v-v_*|}{2}\omega,
\qquad \omega\in\Sph.
\label{eq:post_collision_velocities}
\end{equation}

Here $v_*$ is the velocity of the second colliding particle, $\omega$ is the scattering direction, and $B(v-v_*,\omega)\ge 0$ is the collision kernel.

The kernel may be kept general. Under the isotropic assumption, it has the form
\begin{equation}
B(v-v_*,\omega)=B(|v-v_*|,\cos\theta),
\qquad
\cos\theta=\omega\cdot\frac{v-v_*}{|v-v_*|}.
\label{eq:isotropic_kernel}
\end{equation}

\begin{center}
\begin{tikzpicture}[scale=1.1, line width=1.2pt, >=stealth]

\fill (0,0) circle (2pt);

\draw[->] (-2,1.2) -- (0,0);
\draw[->] (-2,-1.5) -- (0,0);

\draw[->] (0,0) -- (1.8,1.7);
\draw[->] (0,0) -- (2,-0.8);

\draw[->, dashed] (0,0) -- (1.1,0.45);
\node at (1.28,0.62) {$\omega$};

\node at (-2.25,1.35) {$v_*$};
\node at (-2.25,-1.7) {$v$};
\node at (2.05,1.85) {$v_*'$};
\node at (2.25,-0.9) {$v'$};

\end{tikzpicture}
\end{center}

\begin{lemma}
The map
$(v,v_*,\omega)\mapsto (v',v_*',\omega')$ on
$\big((\R^3\times\R^3)\setminus\{(v,v_*):v=v_*\}\big)\times\Sph$ defined by
\begin{equation}
v'=\frac{v+v_*}{2}+\frac{|v-v_*|}{2}\omega,
\qquad
v_*'=\frac{v+v_*}{2}-\frac{|v-v_*|}{2}\omega,
\qquad
\omega'=\frac{v-v_*}{|v-v_*|}
\label{eq:collision_involution}
\end{equation}
is an involution. Its Jacobian determinant has absolute value one, so $\dd\omega'\dd v_*'\dd v'=\dd\omega\dd v_*\dd v$.
\end{lemma}

\begin{proof}
From the first two relations in \eqref{eq:collision_involution},
\[
\frac{v'+v_*'}{2}=\frac{v+v_*}{2},
\qquad
\frac{v'-v_*'}{2}=\frac{|v-v_*|}{2}\omega.
\]
Hence $|v'-v_*'|=|v-v_*|$. Moreover,
\[
v-v_*=|v-v_*|\omega'=|v'-v_*'|\omega'.
\]
Therefore,
\[
v=\frac{v'+v_*'}{2}+\frac{|v'-v_*'|}{2}\omega',
\qquad
v_*=\frac{v'+v_*'}{2}-\frac{|v'-v_*'|}{2}\omega',
\]
and
\[
\omega=\frac{v'-v_*'}{|v'-v_*'|}.
\]
Thus applying the same transformation again recovers the original variables. Hence the map is an involution. The standard elastic-collision change of variables is measure preserving, giving the stated Jacobian identity.
\end{proof}

\paragraph{Normalization.}
In this paper, we treat $f$ as a probability density:
\begin{equation}
\int_{\R^3}\int_{\R^3} f(x,v,t)\dd v\dd x=1,
\qquad 0\le t\le T.
\label{eq:normalization}
\end{equation}
If $f_{\rm phys}=  M p$ is a physical number density with total particle number $  M$, where $p$ is normalized, then
\[
Q[f_{\rm phys},f_{\rm phys}]=  M^2 Q[p,p].
\]
After dividing the physical equation by $\mathcal M$, the normalized equation for $p$ carries a factor $\mathcal M$ on the collision term:
\[
\partial_t p+v\cdot\nabla_x p+\frac{F}{m}\cdot\nabla_v p
=
  M Q[p,p].
\]
Since the collision operator is linear in the kernel $B$, this constant may be absorbed into an effective collision kernel. For simplicity, we continue to denote the effective kernel by $B$.

\section{Weak form of the Boltzmann Equation under external force}
\label{sec:weakform}
Let $\ph=\ph(x,v,t)$ be a smooth admissible test function. For the non-conservative form \eqref{eq:forced_bte_nc}, multiply by $\ph$ and integrate over space, velocity, and time:
\begin{equation}
\int_0^T\int_{\R^6}
\ph\left(\partial_t f+v\cdot\nabla_x f+\frac{F}{m}\cdot\nabla_v f\right)
\dd v\dd x\dd t
=
\int_0^T
\underbrace{\int_{\R^6}\ph Q[f,f]\dd v\dd x}_{=:I(t)}
\dd t.
\label{eq:weak_multiply}
\end{equation}
Integrating by parts in $t$, $x$, and $v$, and assuming that all boundary terms vanish, gives
\begin{equation}
\E_{f(T)}[\ph(\cdot,\cdot,T)]
-
\E_{f(0)}[\ph(\cdot,\cdot,0)]
-
\int_0^T \E_{f(t)}[L^*_{\rm nc}\ph]\dd t
=
\int_0^T I(t)\dd t,
\label{eq:weak_identity_nc}
\end{equation}
where the adjoint operator for the non-conservative form is
\begin{equation}
L^*_{\rm nc}\ph
=
\partial_t\ph+v\cdot\nabla_x\ph+\frac{F}{m}\cdot\nabla_v\ph
+\ph\,\nabla_v\cdot\left(\frac{F}{m}\right).
\label{eq:adjoint_nc}
\end{equation}
For the conservative flux form \eqref{eq:forced_bte_c}, the adjoint is
\begin{equation}
L^*_{\rm c}\ph
=
\partial_t\ph+v\cdot\nabla_x\ph+\frac{F}{m}\cdot\nabla_v\ph.
\label{eq:adjoint_c}
\end{equation}
When \eqref{eq:force_divergence_free} holds, the two adjoints coincide. In all cases, the force contribution appears as a single-density expectation involving
\[
\E_f\left[\frac{F}{m}\cdot\nabla_v\ph\right].
\]
It does not introduce the quadratic density structure that appears in the collision term.

\subsection{Integration in the single-point velocity space}

Fix $x$ and $t$, and let $g(v|x,t)$ denote the conditional velocity density satisfying
\[
\int_{\R^3} g(v|x,t)\dd v=1.
\]
For simplicity, write $g(v|x,t)=g(v)$ and $\ph(x,v,t)=\ph(v)$ in this subsection. The weak collision contribution in velocity space is
\begin{equation}
I_v=\int_{\R^3} Q[g,g](v)\ph(v)\dd v.
\label{eq:Iv_def}
\end{equation}
Splitting the collision operator into gain and loss terms gives
\begin{align}
I_v
&=
\int_{\R^6\times\Sph}
B(v-v_*,\omega)g(v')g(v_*')\ph(v)
\dd\omega\dd v_*\dd v
\notag\\
&\quad -
\int_{\R^6\times\Sph}
B(v-v_*,\omega)g(v)g(v_*)\ph(v)
\dd\omega\dd v_*\dd v.
\label{eq:Iv_gain_loss}
\end{align}
Applying the involution to the gain term, the old relative velocity becomes $|v-v_*|\omega$, and the old scattering direction becomes $(v-v_*)/|v-v_*|$. After relabelling variables, $\ph(v)$ becomes $\ph(v')$. Hence, for a general kernel,
\begin{align}
I_v
&=
\int_{\R^6\times\Sph}
g(v)g(v_*)
\Bigg[
B\left(|v-v_*|\omega,\frac{v-v_*}{|v-v_*|}\right)\ph(v')
-
B(v-v_*,\omega)\ph(v)
\Bigg]
\dd\omega\dd v_*\dd v.
\label{eq:Iv_general_kernel}
\end{align}
Under the isotropic assumption \eqref{eq:isotropic_kernel}, the first kernel factor equals $B(v-v_*,\omega)$. Thus \eqref{eq:Iv_general_kernel} reduces to
\begin{equation}
I_v
=
\int_{\R^6\times\Sph}
B(v-v_*,\omega)g(v)g(v_*)
\big[\ph(v')-\ph(v)\big]
\dd\omega\dd v_*\dd v.
\label{eq:Iv_isotropic}
\end{equation}
Since $|\Sph|=4\pi$,
\[
\int_{\Sph}(\cdot)\dd\omega
=
4\pi\,\E_{\omega\sim U(\Sph)}[(\cdot)].
\]
The factor $4\pi$ may be kept explicitly or absorbed into the collision kernel.

\subsection{Splitting $f$ with respect to $(x,v)$}

Write the joint density as
\begin{equation}
f(x,v,t)=f_x(x,t)g(v|x,t),
\label{eq:fx_split}
\end{equation}
where
\[
f_x(x,t)=\int_{\R^3}f(x,v,t)\dd v,
\qquad
\int_{\R^3}g(v|x,t)\dd v=1.
\]
Then
\[
f(x,v,t)f(x,v_*,t)=f_x^2(x,t)g(v|x,t)g(v_*|x,t).
\]
Therefore, the collision contribution can be written as
\begin{align}
I(t)
&=
\E_{x\sim f_x(\cdot,t)}
\Bigg[
f_x(x,t)
\E_{\substack{v,v_*\sim g(\cdot|x,t)\\ \omega\sim U(\Sph)}}
\Bigg[
B\left(|v-v_*|\omega,\frac{v-v_*}{|v-v_*|}\right)\ph(x,v',t)
\notag\\
&\hspace{4.2cm}
-
B(v-v_*,\omega)\ph(x,v,t)
\Bigg]
\Bigg].
\label{eq:I_split_general}
\end{align}
Under the isotropic assumption, this becomes
\begin{equation}
I(t)
=
\E_{x\sim f_x(\cdot,t)}
\left[
f_x(x,t)
\E_{\substack{v,v_*\sim g(\cdot|x,t)\\ \omega\sim U(\Sph)}}
\left[
B(v-v_*,\omega)\big(\ph(x,v',t)-\ph(x,v,t)\big)
\right]
\right].
\label{eq:I_split_iso}
\end{equation}
The leading spatial factor in the collision term is $f_x^2$, not simply $f_x$. Therefore, if we sample $x\sim f_x(\cdot,t)$, one additional factor $f_x(x,t)$ remains as an explicit density weight. This is why the spatial sampler must be density-estimable. The pointwise value of $g(v|x,t)$ is not required; only samples $v\sim g(\cdot|x,t)$ are needed.

For a candidate density $f_\theta$ with known initial datum $f_0$, define the complete weak residual
\begin{equation}
R_\theta[\ph]
=
\E_{f_\theta(T)}[\ph(\cdot,\cdot,T)]
-
\E_{f_0}[\ph(\cdot,\cdot,0)]
-
\int_0^T \E_{f_\theta(t)}[L^*\ph]\dd t
-
\int_0^T I_\theta(t)\dd t.
\label{eq:complete_residual}
\end{equation}
Here $L^*$ is either $L^*_{\rm nc}$ or $L^*_{\rm c}$, depending on the force formulation. The exact solution satisfies $R_\theta[\ph]=0$ for all admissible test functions.

\section{WANPM for the Boltzmann equation under external force}
\label{sec:wanpm-forced}
\subsection{Admissible test functions}
Because the force contributes the term $\frac{F}{m}\cdot\nabla_v\ph$, the admissible test space must contain functions $\ph(x,v,t)$ that are differentiable in $v$, as well as in $x$ and $t$. In practice, the adversarial test-function network should include genuine $v$ dependence. Otherwise, if $\ph=\ph(x,t)$, then $\nabla_v\ph=0$ and the force term is not tested. Velocity-independent test functions remain admissible as a subset, but they do not probe the force term; they are still useful for conservation diagnostics.

\subsection{WANPM pushforward construction}

The density split \eqref{eq:fx_split} motivates a two-part pushforward model. First, use a density-estimable spatial sampler
\[
x=F_{\theta_x}(z_x,t),
\qquad z_x\sim \pi_{{\rm base},x}.
\]
Its density is obtained by change of variables:
\begin{equation}
f_{\theta_x}(x,t)
=
\pi_{{\rm base},x}\big(F_{\theta_x}^{-1}(x,t)\big)
\left|\det\frac{\partial F_{\theta_x}^{-1}}{\partial x}\right|.
\label{eq:flow_density}
\end{equation}
Second, use a conditional velocity sampler
\[
v=F_{\theta_v}(z_v;x,t),
\qquad z_v\sim \pi_{{\rm base},v}.
\]
Only samples from this conditional velocity model are required; its pointwise density is not used.

Thus a phase-space sample is generated by
\[
z_x\sim\pi_{{\rm base},x},\qquad x\leftarrow F_{\theta_x}(z_x,t),
\]
and then
\[
z_v\sim\pi_{{\rm base},v},\qquad v\leftarrow F_{\theta_v}(z_v;x,t).
\]
Because both components are pushforwards of a base density through the sampler maps, the model density is nonnegative by construction: $f_{\theta}\ge 0$ pointwise, with no penalty or projection required. This is a structural advantage over residual form collocation methods, in which positivity of $f_{\theta}$ must be forced separately.

\subsection{Objective and Monte Carlo estimator}

In the implemented method the adversarial test space is a bank of $K$ trainable plane waves
\begin{equation}
\ph_k(y,t)=\sin\!\big(w_k\cdot y+\kappa_k\,t+\beta_k\big),\qquad
\eta=\{w_k,\kappa_k,\beta_k\}_{k=1}^{K},\qquad y=(x,v),
\label{eq:planewave}
\end{equation}
whose adjoint action is analytic,
$L^\ast\ph_k=(\kappa_k+b(y)\cdot w_k)\cos(w_k\cdot y+\kappa_k t+\beta_k)$ with $b(y)=(v,-\tfrac{F}{m})$.
Writing $R_k=R_\theta[\ph_k]$ for the weak residual of the $k$-th test function, the pure objective is the
squared weak residual averaged over the bank,
\begin{equation}
\boxed{\;\min_\theta\ \max_\eta\ \ \mathcal L_{\rm pure}(\theta,\eta)=\frac1K\sum_{k=1}^{K}R_k^2\;}
\label{eq:minmax}
\end{equation}
with no additional normalization or penalty. The plane-wave frequencies are  trained adversarially to maximise the residual, while the pushforward parameters $\theta$ minimise it. A Sobolev-type
normalization $\|\ph_\eta\|_Y^2=\E[\ph_\eta^2+|\partial_t\ph_\eta|^2+|\nabla_x\ph_\eta|^2+|\nabla_v\ph_\eta|^2]$
may be used to divide the residual when the test network is unconstrained; for the bounded plane-wave bank it is not required since  $|\ph_k|\le1$ prevents the adversary from inflating the objective by rescaling.

Each expectation in $R_\theta$ is estimated by sample averages. For the transport term, sample $t_i\sim U[0,T]$ and then $(x_i,v_i)\sim f_\theta(\cdot,\cdot,t_i)$. Then
\begin{equation}
\int_0^T \E_{f_\theta(t)}[A(t)]\dd t
\approx
\frac{T}{N}\sum_{i=1}^N A(x_i,v_i,t_i).
\label{eq:mc_time}
\end{equation}
For the collision term, sample $x_i\sim f_{\theta_x}(\cdot,t_i)$, evaluate $f_{\theta_x}(x_i,t_i)$, sample $v_i,v_{*,i}\sim g_{\theta_v}(\cdot|x_i,t_i)$, and sample $\omega_i\sim U(\Sph)$. For an isotropic kernel,
\begin{equation}
\int_0^T I_\theta(t)\dd t
\approx
\frac{T}{N}\sum_{i=1}^N
f_{\theta_x}(x_i,t_i)
B(v_i-v_{*,i},\omega_i)
\left[\ph_\eta(x_i,v_i',t_i)-\ph_\eta(x_i,v_i,t_i)\right],
\label{eq:mc_collision}
\end{equation}
where
\[
v_i'=\frac{v_i+v_{*,i}}{2}+\frac{|v_i-v_{*,i}|}{2}\omega_i.
\]
If the factor $4\pi$ from the surface measure of $\Sph$ is not absorbed into $B$, then the right-hand side of \eqref{eq:mc_collision} should be multiplied by $4\pi$.

\textbf{Initial Condition:}
If the sampler does not enforce $f_\theta(\cdot,\cdot,0)=f_0$ by construction, the initial condition should be imposed through a distributional discrepancy, for example maximum mean discrepancy, Wasserstein distance, or an adversarial initial-condition loss
\[
\mathcal L_{\rm IC}
=
\sup_{\psi\in\Psi}
\left|\E_{f_\theta(\cdot,\cdot,0)}[\psi]-\E_{f_0}[\psi]\right|^2.
\]
Because the conditional velocity component is sample-only, the full joint pointwise density $f_\theta(x,v,t)$ is generally unavailable. Therefore, the default initial-condition loss should be distributional rather than pointwise.

\subsection{Conservative quantities}

For elastic Boltzmann collisions with the standard symmetry assumptions on the collision kernel, the collision invariants are
\[
\ph(v)\in\{1,\ v_i\ (i=1,2,3),\ |v|^2\}.
\]
For these choices, the collision contribution vanishes. In the force-free case, the weak form expresses conservation of mass, momentum, and kinetic energy. With an external force, the same choices give transport-force balance laws. Mass remains conserved for the conservative force formulation, or when $\nabla_v\cdot(F/m)=0$, while momentum and energy may change through the force term.
\subsection{Polynomial-augmented objective}
\label{sec:poly_obj}
The plane-wave bank senses the covariance of the pushforward only through the scalar $w^\top\Sigma w$, whose off-diagonal contributions carry random signs and partially cancel; low-order phase-space moments are therefore weakly constrained by the sine bank alone. To pin them directly, we optionally append a fixed bank of second-order test functions $\varphi\in\ x_i^2,\,v_i^2,\,x_iv_i\}$ whose weak residuals are enforced  cumulatively at every quadrature node,
\begin{equation}
R^{\varphi}_m=\E_{f(t_m)}[\varphi]-\E_{f_0}[\varphi]-\int_0^{t_m}\E[L^\ast\varphi]\,\dd t-\int_0^{t_m}I^{\rm coll}_\varphi\,\dd t .
\label{eq:poly_res}
\end{equation}
Because $L^\ast(x_iv_i)=v_i^2-\tfrac{F_i}{m}x_i$ reads the position--velocity covariance directly, this bank supplies
the signal the sine bank cancels. With $n_s$ sine and $n_p$ polynomial test functions the augmented objective is
\begin{equation}
\boxed{\;\min_\theta\ \max_\eta\ \
\mathcal L_{\rm poly}=\underbrace{\frac1{n_s}\sum_{k=1}^{n_s}R_k^2}_{\text{adversarial sine}}
+\lambda_{\rm poly}\underbrace{\frac1{n_p(Q-1)}\sum_{m=1}^{Q-1}\sum_{\varphi}\big(R^{\varphi}_m\big)^2}_{\text{fixed polynomial anchor}}\;}
\label{eq:minmax_poly}
\end{equation}
with weight $\lambda_{\rm poly}\ge0$; setting $\lambda_{\rm poly}=0$ recovers the pure objective \eqref{eq:minmax}
exactly. The polynomial bank carries no trainable parameters, so it adds constraints without enlarging the adversary.

\subsection{WANPM architecture}
\label{sec:wanpm_arch}
The pushforward is a time-conditioned invertible map built from the density split \eqref{eq:fx_split}: a
density-estimable spatial flow $x=F_{\theta_x}(t,x_0)$ that returns its log-Jacobian, and a sample-only conditional
velocity flow $v=F_{\theta_v}(t,x,v_0)$ conditioned on the pushed position. Each is a stack of affine coupling
layers whose scale and shift are produced by an MLP acting on the frozen (conditioning) block, and every layer is
multiplied by a temporal gate $g(t)$ with $g(0)=0$, so that
\begin{equation}
F_\theta(0,\cdot)=\mathrm{id}\quad\Longrightarrow\quad \rho_0=f_0\ \text{ exactly,}
\label{eq:gate_id}
\end{equation}
and the initial condition holds by construction with no IC penalty. A single coupling layer maps, per transformed
coordinate $j$,
\begin{equation}
y^{(\ell)}_j=y^{(\ell-1)}_j\exp\!\big(g(t)\,s_{\max}\tanh s^{(\ell)}_j\big)+g(t)\,u^{(\ell)}_j,\qquad
(s^{(\ell)},u^{(\ell)})=\mathrm{MLP}_\ell\big([\,y^{(\ell-1)}_{\rm cond},\,t\,]\big),
\label{eq:coupling}
\end{equation}
with the conditioning and transformed blocks swapped on alternate layers and the final MLP layer zero-initialised so that $F_\theta=\mathrm{id}$ at the start of training. Only the spatial flow returns a log-Jacobian, because the collision weight in \eqref{eq:mc_collision} requires the value of $f_{\theta_x}$; the velocity flow is sample-only. The optimisation is a two-player min--max: the pushforward parameters $\theta$ are trained by Adam (minimiser) and the plane-wave parameters $\eta$ by SGD (maximiser) at a larger learning rate, with $n_{\rm critic}$ ascent steps per
generator step and gradient clipping on both networks. Table~\ref{tab:wanpm_hyper} lists the components.

\begin{table}[H]
\centering
\small
\begin{tabular}{@{}ll@{}}
\toprule
component & choice\\
\midrule
spatial flow $F_{\theta_x}$ & affine-coupling RealNVP, density-estimable (returns $\log|\det|$)\\
velocity flow $F_{\theta_v}$ & affine-coupling RealNVP, conditioned on $x$, sample-only\\
coupling layers & $6$ per flow, alternating split, MLP $\to 64\to 64\to$ scale/shift, $\tanh$\\
temporal gate & $g(t)=\sqrt t$, $g(0)=0$ $\Rightarrow F_\theta(0,\cdot)=\mathrm{id}$\\
log-scale bound & $s_{\max}$ via $\tanh$\\
test bank $\ph_\eta$ & $K$ plane waves $\sin(w_k\!\cdot\! y+\kappa_k t+\beta_k)$, adversarial\\
polynomial anchor & fixed $\{x_i^2,v_i^2,x_iv_i\}$, optional, weight $\lambda_{\rm poly}$\\
generator optimiser & Adam (minimiser)\\
adversary optimiser & SGD (maximiser), larger lr, $n_{\rm critic}$ ascent steps, grad clip\\
time quadrature & clustered nodes $t_j=T(j/(Q-1))^2$ (matched to $\sqrt t$ gate)\\
\bottomrule
\end{tabular}
\caption{WANPM components for the forced Boltzmann equation. The generator is a split pushforward (density-estimable position flow, conditional velocity sampler); the adversary is a trainable plane-wave bank; the optional polynomial
anchor pins the low-order moments.}
\label{tab:wanpm_hyper}
\end{table}

\begin{figure}[htbp]
    \centering
    \includegraphics[width=\textwidth]{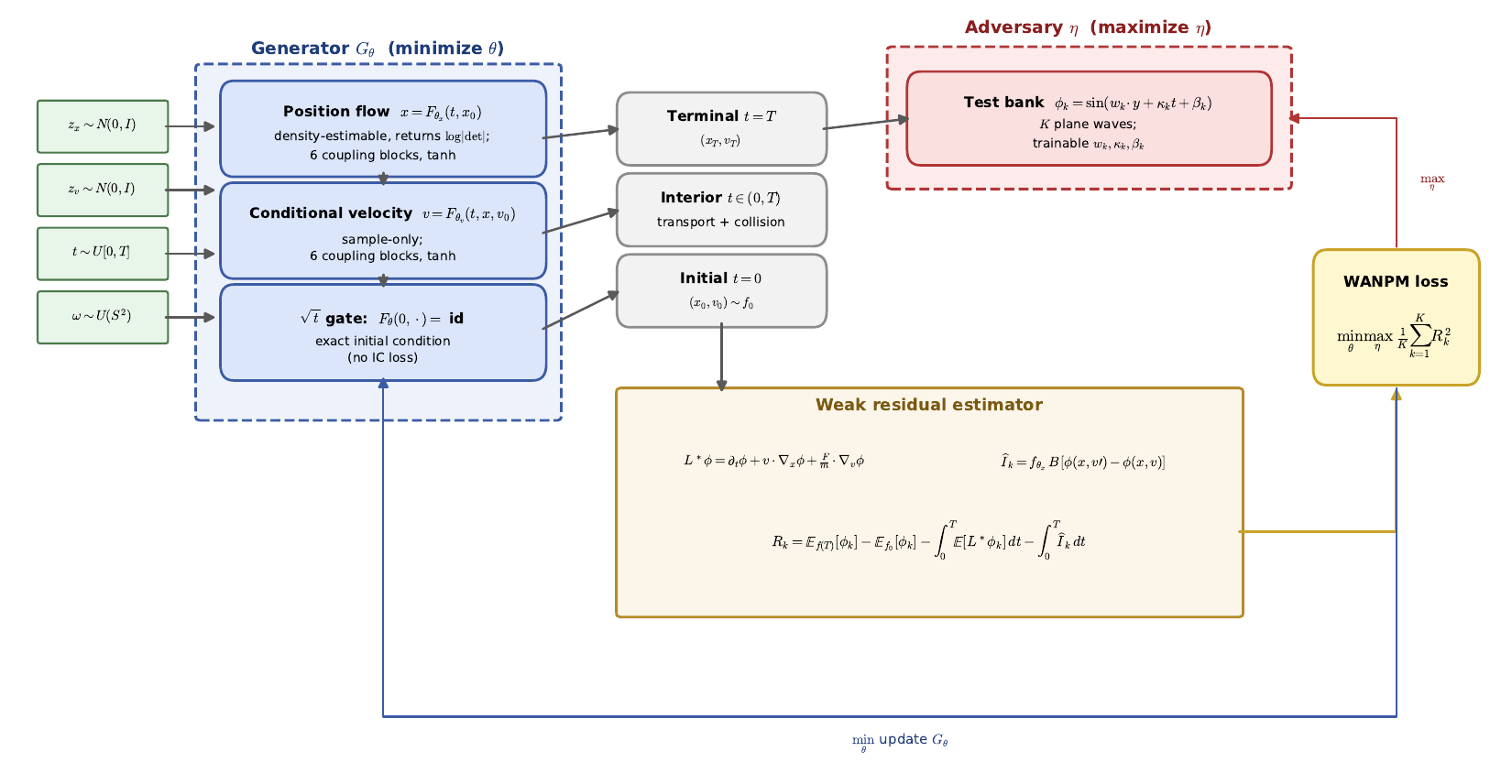}
    \caption{
   WANPM architecture for the forced Boltzmann equation. The generator $G_\theta$ yields phase-space samples using a
 density-estimable position flow $F_{\theta_x}$ and a conditional velocity sampler $F_{\theta_v}$ of Gaussian latent variables, both $\sqrt t$-gated so that $F_\theta(0,\cdot)=\mathrm{id}$. The weak residual estimator assembles terminal, initial, transport/force, and collision terms; the collision term uses the spatial log-Jacobian to recover $f_{\theta_x}$. The adversarial test network $\ph_\eta$ (a plane-wave bank, optionally augmented by a fixed polynomial moment bank) maximises the weak residual while the generator minimises the pure weak-form loss $\mathcal L_{\mathrm{pure}}=\frac1K\sum_k R_k^2$. The initial condition holds by construction through the $\sqrt t$ gate, so no initial-condition or regularization term enters the loss.
    }
    \label{fig:wanpm_architecture}
\end{figure}

\section{Numerical results}
\label{sec:numerical_results}
\subsection{Collisionless free transport}

we consider the collisionless and force-free case. Therefore,
\begin{equation*}
F=0,
\qquad
Q[f,f]=0.
\end{equation*}

Substituting these assumptions into the general BE gives the free-transport equation
\begin{equation}
\boxed{
\partial_t f+v\cdot\nabla_x f=0.
}
\label{eq:free_transport}
\end{equation}

The initial condition is prescribed at $t=0$, hence, the full initial-value problem is
\begin{equation}
\boxed{
\partial_t f+v\cdot\nabla_x f=0,
\qquad
f(x,v,0)=f_0(x,v).
}
\label{eq:ivp_free_transport}
\end{equation}
In the free-transport problem, both the physical space and velocity space are \textbf{three-dimensional}. We write
\begin{equation*}
x=(x_1,x_2,x_3)\in \mathbb R^3,
\qquad
v=(v_1,v_2,v_3)\in \mathbb R^3,
\qquad
t\geq 0.
\end{equation*}

we choose the \textit{initial condition} as a Gaussian distribution in both position and velocity:
\begin{equation}
\boxed{
f_0(x,v)
=
\mathcal N(x;0,\sigma_x^2 I_3)
\mathcal N(v;0,\sigma_v^2 I_3).
}
\label{eq:gaussian_initial_condition}
\end{equation}

Here, $\sigma_x>0$ controls the spread in physical space, and $\sigma_v>0$ controls the spread in velocity space.

\subsubsection{DSMC numerical result}
\label{subsec:dsmc_result}

The DSMC result for Experiment E1 is summarized as follows:
\begin{equation*}
\begin{array}{ll}
\text{Wall-clock time} & 0.19\ \mathrm{s},\\
\text{DSMC time-step iterations} & 200,\\
\mathrm{MSE}_{x_1} & 3.494\times 10^{-6},\\
\mathrm{MAE}_{x_1} & 1.266\times 10^{-3},\\
\mathrm{Rel}\,L^2_{x_1} & 1.332\times 10^{-2},\\
\mathrm{MSE}_{v_1} & 2.042\times 10^{-5},\\
\mathrm{Rel}\,L^2_{v_1} & 2.708\times 10^{-2},\\
\operatorname{Corr}(x_1,v_1) & 0.704.
\end{array}
\end{equation*}

These values show that DSMC reproduces the exact E1 solution with small marginal and full phase-space errors. Since the E1 dynamics are solved exactly by the characteristic streaming update, the remaining discrepancy is mainly due to finite-sample Monte Carlo noise and density-estimation error. Thus, DSMC serves as the natural reference baseline for evaluating the learned PINN and WANPM solvers.

\subsubsection{WANPM: model and implementation}
\label{sec:wanpm}

WANPM represents the time-dependent solution by a pushforward generative map. Instead of approximating the density on a fixed six-dimensional grid, the method learns a map that transforms Gaussian latent variables into samples from \(f(\cdot,\cdot,t)\). The model is trained on the \emph{pure} weak-form residual in a min--max game: the generator drives the weak residual against a bank of adversarial plane-wave test functions to zero, while the test bank searches for the frequencies that most expose it. No initial-condition, moment, or covariance penalty is added; the initial condition holds by construction through a temporal gate (below).

\paragraph{Network architecture:}

The WANPM generator \(G_\theta\) maps Gaussian latent variables to phase-space samples:  $G_\theta(z,t)=(x_\theta(z,t),v_\theta(z,t)).$ The generator is factorized into a position flow and a velocity sampler. The position component is modeled by a RealNVP normalizing flow,
$$
x=\mathrm{flow}_{\theta}(z_x,t),
$$
with six affine coupling blocks. Each coupling block is time-conditioned and uses MLP coupling networks with \textbf{tanh} activations. This flow gives an exact log-density contribution through its Jacobian. The velocity component is generated conditionally on position:
$$
v=\mathrm{vel}_\theta(z_v,x,t).
$$
In the implementation, the conditional velocity component is itself an affine-coupling flow with six blocks whose
scale and shift are produced by MLPs of width $64$ with \textbf{tanh} activations, conditioned on $[x,t]$. Every
coupling layer is multiplied by the temporal gate $g(t)=\sqrt t$, so that at $t=0$ the map is the identity and the
velocity distribution equals the Gaussian reference exactly; only samples from this conditional flow are used, its
pointwise density is not required. The adversarial test space is a bank of $K$ trainable plane waves
\[
\phi_k(x,v,t)=\sin\!\big(w_k\cdot(x,v)+\kappa_k\,t+\beta_k\big),\qquad \eta=\{w_k,\kappa_k,\beta_k\}_{k=1}^{K},
\]
whose adjoint action is analytic, so no automatic differentiation of the test network is required. The frequencies
$w_k$ are initialized on a band that fixes the second-moment sensitivity peak and are then trained by gradient ascent to maximize the residual. The generator $\theta$ (position flow plus conditional velocity flow) is trained by gradient descent to minimise it.

\paragraph{Initial condition:}

The initial condition is enforced \emph{by construction}. Because every coupling layer is multiplied by the gate
$g(t)=\sqrt t$ with $g(0)=0$, the generator reduces to the identity at $t=0$,
$$
F_\theta(0,\cdot)=\mathrm{id}\qquad\Longrightarrow\qquad \rho_0=f_0\ \text{ exactly,}
$$
for every parameter value. No initial-condition penalty is therefore required, and the corresponding diagnostic is
identically zero throughout training.



\paragraph{Optimizer:}

The generator (position and velocity flows) is trained with Adam and the plane-wave test bank with a separate optimiser at a larger learning rate, a two-timescale ratio that keeps the min-max stable. Gradient clipping with maximum norm \(1.0\) is applied to both networks, $n_{\mathrm{critic}}$ ascent steps on the test bank precede each generator step, and the learning rates are annealed toward $\eta_{\min}=10^{-6}$. These choices stabilise the adversarial training; the weak residual is expected to be non-monotone (adversarial), and its oscillation is not a sign of divergence.

\paragraph{Loss function (pure WANPM):}

The generator minimizes, while the adversarial test bank maximizes, the pure weak-form objective: the squared weak
residual averaged over the $K$ plane-wave test functions,
\begin{equation}
\boxed{\;\mathcal L_{\mathrm{pure}}(\theta,\eta)
=\frac{1}{K}\sum_{k=1}^{K}R_k^2\;}
\label{eq:wanpm_total_loss}
\end{equation}
with no initial condition, moment, or covariance penalty. The initial condition is enforced by construction through the
$\sqrt t$ gate ($F_\theta(0,\cdot)=\mathrm{id}$, so $\rho_0=f_0$ exactly), and the low-order moments are left to the weak residual itself; the optional polynomial anchor of Section~\ref{sec:poly_obj} is disabled here ($\lambda_{\mathrm{poly}}=0$). The residual $R_k$ is defined next.

\paragraph{Weak residual.}
The weak residual is obtained by multiplying the transport equation by a smooth test function \(\phi\), integrating over phase space and time, and integrating by parts. For E1, the acceleration and collision terms vanish, so the weak identity becomes
\begin{equation}
\mathbb E_{f(T)}[\phi_T]
-
\mathbb E_{f(0)}[\phi_0]
-
\int_0^T
\mathbb E_{f(t)}
\left[
\partial_t\phi+v\cdot\nabla_x\phi
\right]
\,dt
=
0.
\label{eq:wanpm_weak_identity_e1}
\end{equation}
Since the generator provides samples from the candidate distribution, each expectation is approximated by \textit{Monte Carlo averages} over generated samples. For each adversarial test function \(\phi_k\), the residual is approximated by
\begin{equation}
R_k
=
\mathbb E[\phi_k(x_T,v_T,T)]
-
\mathbb E[\phi_k(x_0,v_0,0)]
-
T\mathbb E
\left[
\partial_t\phi_k
+
v\cdot\nabla_x\phi_k
\right].
\label{eq:wanpm_residual}
\end{equation}
The pure weak loss is the mean squared residual over the bank,
\begin{equation}
\mathcal L_{\mathrm{pure}}
=
\frac{1}{K}
\sum_{k=1}^{K}
R_k^2,
\label{eq:wanpm_weak_loss}
\end{equation}
with no residual normalization, since the plane-wave test functions are bounded ($|\phi_k|\le1$) and the adversary
cannot inflate the objective by rescaling. The generator drives \(R_k\to 0\) while the adversary trains the
frequencies $\{w_k,\kappa_k,\beta_k\}$ to make the residual large.

\paragraph{How  the WANPM error is measured:}

WANPM has two different notions of error. During training, the equation error is measured by the weak residual loss \(\mathcal L_{\mathrm{weak}}\). This checks whether the generated distribution satisfies the transport equation in weak form. At evaluation time, WANPM is assessed in the same way as the particle solvers. A final-time sample set is drawn from the trained generator: $\left\{
G_\theta(z_j,T)
\right\}_{j=1}^{N}.$
The one-dimensional marginals are estimated by the same KDE procedure used for DSMC, and the marginal MSE, MAE, and relative \(L^2\) errors are computed against the exact Gaussian marginals. Because one-dimensional marginals cannot detect correlations between \(x\) and \(v\), the joint phase-space accuracy is assessed directly through the second moments---the per-axis variances and, crucially, the position--velocity covariance \(\operatorname{Cov}(x_k,v_k)\) and correlation, which the free-streaming tilt makes the decisive diagnostic. Thus the weak residual certifies that the learned pushforward satisfies the physics, while the marginal errors and the phase-space covariance measure the accuracy of the output distribution.

\subsubsection{PINN: model and implementation}
\label{sec:pinn}

The PINN baseline approximates the density field directly as a function of the full space-velocity-time input,
$$
f_\theta=f_\theta(x,v,t),
\qquad
(x,v,t)\in\mathbb R^3\times\mathbb R^3\times[0,T].
$$
Unlike DSMC and WANPM, which are sample-based methods, the PINN learns a pointwise density. The transport equation is enforced by minimizing the strong-form PDE residual at randomly sampled collocation points.

\paragraph{Network architecture:}
The PINN uses a fully connected neural network with input
$
[x,v,t]\in\mathbb R^7
$
and scalar output. The architecture is
$$
\mathbb R^7
\longrightarrow
128
\longrightarrow
128
\longrightarrow
128
\longrightarrow
128
\longrightarrow
1,
$$
with tanh activations. To enforce non-negativity and correct tail decay, the raw network output is passed through a softplus function and multiplied by a fixed Gaussian envelope:
\begin{equation}
f_\theta(x,v,t)
=
\operatorname{softplus}(N_\theta(x,v,t))\,E(x,v).
\label{eq:pinn_density}
\end{equation}
Here,
\begin{equation}
E(x,v)
=
\exp\left(
-\frac{|x|^2}{2\sigma_{\mathrm{env},x}^2}
-\frac{|v|^2}{2\sigma_{\mathrm{env},v}^2}
\right),
\label{eq:pinn_envelope}
\end{equation}
with
$$
\sigma_{\mathrm{env},x}^2=\sigma_x^2+T^2\sigma_v^2,
\qquad
\sigma_{\mathrm{env},v}^2=\sigma_v^2.
$$
For E1, where \(\sigma_x=\sigma_v=1\) and \(T=1\), this gives
$$
\sigma_{\mathrm{env},x}^2=2,
\qquad
\sigma_{\mathrm{env},v}^2=1.
$$

\paragraph{How the PINN uses the equation:}
The PINN enforces this equation in strong form. The residual is defined by
\begin{equation}
r_\theta(x,v,t)
=
\partial_t f_\theta(x,v,t)
+
v\cdot\nabla_x f_\theta(x,v,t).
\label{eq:pinn_residual}
\end{equation}
all derivatives are computed by automatic differentiation. For a general forced problem, the residual would include \(a\cdot\nabla_v f_\theta\), but for E1 the acceleration is $a=0.$ Therefore, only the time derivative and spatial transport derivative are needed. The collocation points are sampled from broad Gaussian proposal distributions in \(x\) and \(v\), together with uniformly sampled time:
$$
x_c\sim 4\,\mathcal N(0,I_3),
\qquad
v_c\sim 3\,\mathcal N(0,I_3),
\qquad
t_c\sim\mathcal U(0,T).
$$
This gives the PINN training points in the full seven-dimensional input domain.

\paragraph{Loss function:}
\label{subsec:pinn_loss}

The PINN loss contains a PDE residual term and an initial-condition term:
\begin{equation}
\mathcal L_{\mathrm{PINN}}
=
\mathcal L_{\mathrm{pde}}
+
10\,\mathcal L_{\mathrm{ic}}.
\label{eq:pinn_total_loss}
\end{equation}
The PDE loss is
\begin{equation}
\mathcal L_{\mathrm{pde}}
=
\frac{1}{N_c}
\sum_{c=1}^{N_c}
r_\theta(x_c,v_c,t_c)^2.
\label{eq:pinn_pde_loss}
\end{equation}
This term penalizes violation of the transport equation. The initial-condition loss is
\begin{equation}
\mathcal L_{\mathrm{ic}}
=
\frac{1}{N_0}
\sum_{i=1}^{N_0}
\left(
f_\theta(x_0^{(i)},v_0^{(i)},0)
-
f_0(x_0^{(i)},v_0^{(i)})
\right)^2.
\label{eq:pinn_ic_loss}
\end{equation}
Here, the exact initial density is
\begin{equation}
f_0(x,v)
=
\frac{1}{(2\pi)^3}
\exp\left(
-\frac{|x|^2+|v|^2}{2}
\right).
\label{eq:pinn_initial_density}
\end{equation}
The IC term anchors the neural density to the known Gaussian initial condition at \(t=0\).

\paragraph{Optimizer and training:}
\label{subsec:pinn_optimizer}

The PINN is trained using AdamW with learning rate
$
10^{-3},
$
small weight decay, and batch size \(4096\). In the original E1 run, the PINN was trained for \(50{,}000\) iterations to match the WANPM training budget. However, this was found to over-train the spatial marginal. In the corrected implementation, the PINN and WANPM iteration counts are decoupled, and the PINN is trained in its stable regime, around
$
20{,}000
$
iterations. The important distinction is that PINN iterations are optimization steps, not physical time steps. The network is not marching particles in time; instead, it learns a global density function \(f_\theta(x,v,t)\) over the full time interval.

\paragraph{Evaluation procedur:e}
\label{subsec:pinn_evaluation}

The PINN training loss measures the strong-form PDE and IC residuals, but the final solver comparison uses the same output metrics as DSMC and WANPM. To evaluate the PINN at \(t=T\), samples are drawn from the learned density
$$
f_\theta(x,v,T)
=
\operatorname{softplus}(N_\theta(x,v,T))\,E(x,v)
$$
using rejection sampling with the Gaussian envelope \(E(x,v)\) as the proposal distribution. These samples are then processed exactly like the DSMC and WANPM samples. The one-dimensional marginals are estimated using the same KDE procedure, and the errors are computed using the same metrics:
$$
\mathrm{MSE},
\qquad
\mathrm{MAE},
\qquad
\mathrm{Rel}\,L^2.
$$
The joint phase-space accuracy is measured through the second moments, in particular the position--velocity
covariance $\operatorname{Cov}(x_k,v_k)$ and correlation, which the free-streaming tilt makes the decisive
diagnostic. Thus, the PINN training error and the PINN evaluation error are different quantities. The training loss checks whether the density satisfies the PDE at collocation points, while the evaluation metrics check whether the learned density produces the correct final-time distribution.

\subsubsection{Results}
\label{sec:results}

Table~\ref{tab:summary1} collects the final-time metrics for all solvers against the exact
solution . Errors are reported on the $x_1$ and $v_1$ marginals (relative
$L^2$, MSE, MAE) and on the joint phase space via the position--velocity covariance and correlation.
The Exact row is zero by construction and is included for reference.

\begin{table}[H]
\centering
\caption{E1 free transport at $t=1$: solver comparison against the exact solution.
The joint phase-space accuracy is reported separately through the covariance in Table~\ref{tab:corr}.}
\label{tab:summary1}
\small
\begin{tabular}{lrrrrrr}
\toprule
Solver & $\mathrm{MSE}_{x_1}$ & $\mathrm{MAE}_{x_1}$ & $\relL_{x_1}$ & $\mathrm{MSE}_{v_1}$ & $\relL_{v_1}$ & Iters\\
\midrule
Exact  & 0 & 0 & 0 & 0 & 0 & --\\
DSMC   & $1.86\!\times\!10^{-6}$ & $1.10\!\times\!10^{-3}$ & $9.72\!\times\!10^{-3}$ & $2.13\!\times\!10^{-5}$ & $2.76\!\times\!10^{-2}$ & $200$\\
PINN   & $1.97\!\times\!10^{-5}$ & $3.21\!\times\!10^{-3}$ & $3.17\!\times\!10^{-2}$ & $1.35\!\times\!10^{-5}$ & $2.20\!\times\!10^{-2}$ & $20{,}000$\\
WANPM  & $5.55\!\times\!10^{-6}$ & $1.79\!\times\!10^{-3}$ & $1.68\!\times\!10^{-2}$ & $2.09\!\times\!10^{-5}$ & $2.74\!\times\!10^{-2}$ & $50{,}000$\\
\bottomrule
\end{tabular}
\end{table}
Here DSMC is the most accurate solver, as expected: free transport is a pure characteristic problem that the particle update solves exactly, so it sets the finite-sample Monte~Carlo floor at $N=10^4$. WANPM reproduces the exact distribution closely, with balanced marginal errors ($\relL_{x_1}=1.68\times10^{-2}$, $\relL_{v_1}=2.74\times10^{-2}$) comparable to DSMC's. The PINN, trained in its stable regime (\texttt{cfg.pinn\_steps}$=20{,}000$), also has accurate \emph{marginals} ($\relL_{x_1}=3.2\times10^{-2}$, $\relL_{v_1}=2.2\times10^{-2}$); the difference between the solvers is not visible in the marginals at all, but in the joint phase-space structure measured by the covariance (Table~\ref{tab:corr}). It would be incorrect to claim WANPM is more accurate than DSMC on E1; the value of WANPM is that it is a learned, differentiable sampler that captures the joint phase-space structure and extends to the forced and collisional regimes where DSMC's exactness no longer holds.

\paragraph{Phase-space correlation: }

The marginal errors hide the real difference between the solvers. Free transport \emph{creates} a position-velocity correlation: a particle starting at $\vx_0$ arrives at $\vx_0+t\vv_0$, so by $t=1$ the exact joint has $\operatorname{Cov}(x_k,v_k)=t\sigma_v^2=1$ -- the $45^\circ$ tilt of the phase-space ellipse. Table~\ref{tab:corr} measures this tilt directly from each solver's final samples on the $(x_1,v_1)$ slice.

\begin{table}[H]
\centering
\caption{Phase-space correlation at $t=1$ on the $(x_1,v_1)$ slice, from $4\times10^4$ samples.
All three solvers recover the marginals (Var$(x_1)\!\approx\!2$, Var$(v_1)\!\approx\!1$), but
only DSMC and WANPM recover the correlation $\operatorname{Cov}(x_1,v_1)\!\approx\!1$; the PINN's
is essentially zero, so its joint distribution is untilted.}
\label{tab:corr}
\small
\begin{tabular}{lrrrr}
\toprule
Solver & Var$(x_1)$ & Var$(v_1)$ & $\operatorname{Cov}(x_1,v_1)$ & corr\\
\midrule
Exact  & $2.000$ & $1.000$ & $1.000$  & $0.707$\\
DSMC   & $1.976$ & $0.997$ & $0.988$  & $0.704$\\
PINN   & $2.038$ & $1.005$ & $\textbf{0.000}$  & $\textbf{0.000}$\\
WANPM  & $2.007$ & $1.021$ & $1.012$  & $0.707$\\
\bottomrule
\end{tabular}
\end{table}

This is exactly what Figure~\ref{fig:e12} shows: the Exact, DSMC, and WANPM panels are tilted
ellipses, while the PINN panel is an upright, uncorrelated blob.

\paragraph{Why the PINN is worse.}
The PINN setup is not mathematically wrong: the residual
\(\partial_t f+\vv\cdot\grad_\vx f\) and the initial condition are correct, and the exact tilted Gaussian is the unique solution. The issue is numerical conditioning and representation. In practice, the PINN converges to a marginally correct but nearly uncorrelated product-like density,
$$
p(\vx,\vv,T)\approx \mathcal N(\vx;0,2I)\mathcal N(\vv;0,I),
$$
instead of the true coupled density \(f_0(\vx-T\vv,\vv)\). This happens because:
\begin{itemize}
    \item the initial condition at \(t=0\) is uncorrelated;
    \item the positivity envelope \(E(\vx,\vv)\) is diagonal and time-independent, so the network must create the full tilt by itself;
    \item independently sampled collocation points do not strongly enforce the correlated ridge \(\vx\approx t\vv\);
    \item more iterations do not remove this defect: the covariance stays essentially zero regardless of the iteration budget, so the joint density remains untilted.
\end{itemize}
Thus, the PINN learns good one-dimensional marginals but misses the joint \(x\)-\(v\) covariance.

\paragraph{Why WANPM is better suited to transport.}
WANPM represents the solution as a pushforward sampler,
$$
(\vx,\vv)=G_\theta(\vz,t),
$$
rather than as a pointwise density field. This gives two advantages:
\begin{itemize}
    \item the generator can directly learn joint phase-space structure, including the tilted \(x\)-\(v\) correlation;
    \item the weak adversarial residual tests the transport equation in distributional form, so correlation-free candidates are exposed by suitable test functions.
\end{itemize}
As a result, WANPM captures the tilted phase-space distribution much better than the density-PINN. DSMC also captures the tilt because it moves particles along exact characteristics, but DSMC is not a learned differentiable model and does not provide the same framework for the later forced and collisional BTE cases.

\subsubsection{Figures}

\begin{figure}[htbp]
\centering
\includegraphics[width=0.62\textheight]{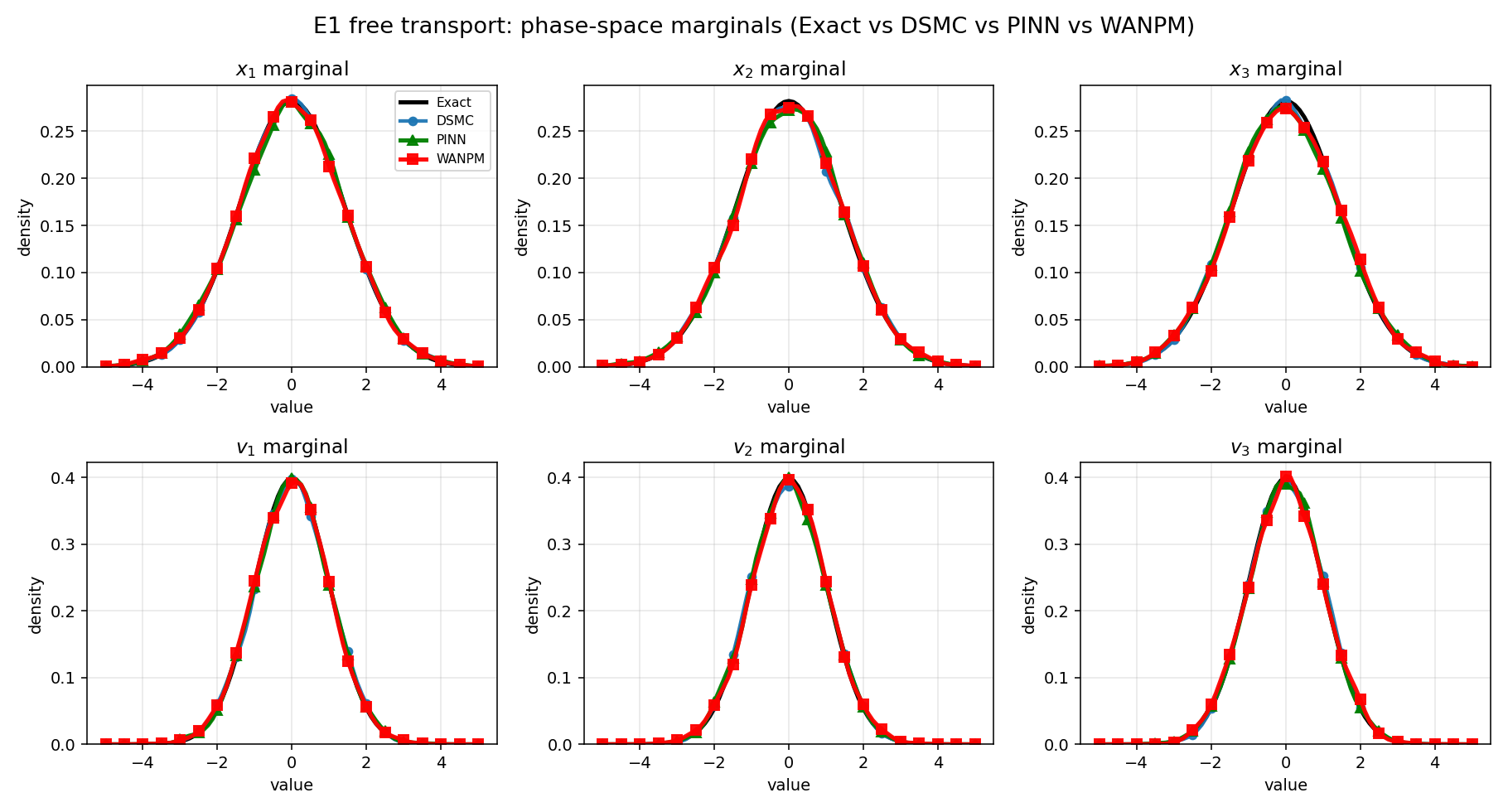}
\caption{\textbf{Six one-dimensional marginals at $t=1$.} Spatial marginals $x_1,x_2,x_3$ (top) and velocity marginals $v_1,v_2,v_3$ (bottom), comparing the exact solution (thick black) with DSMC, PINN, and WANPM. The spatial marginals are wider ($\operatorname{Var}=2$) than the velocity marginals ($\operatorname{Var}=1$), reflecting ballistic spreading. DSMC and WANPM track the exact curves; the PINN spatial marginals show the over-training distortion noted above.}
\label{fig:e11}
\end{figure}

\begin{figure}[htbp]
\centering
\includegraphics[width=0.62\textheight]{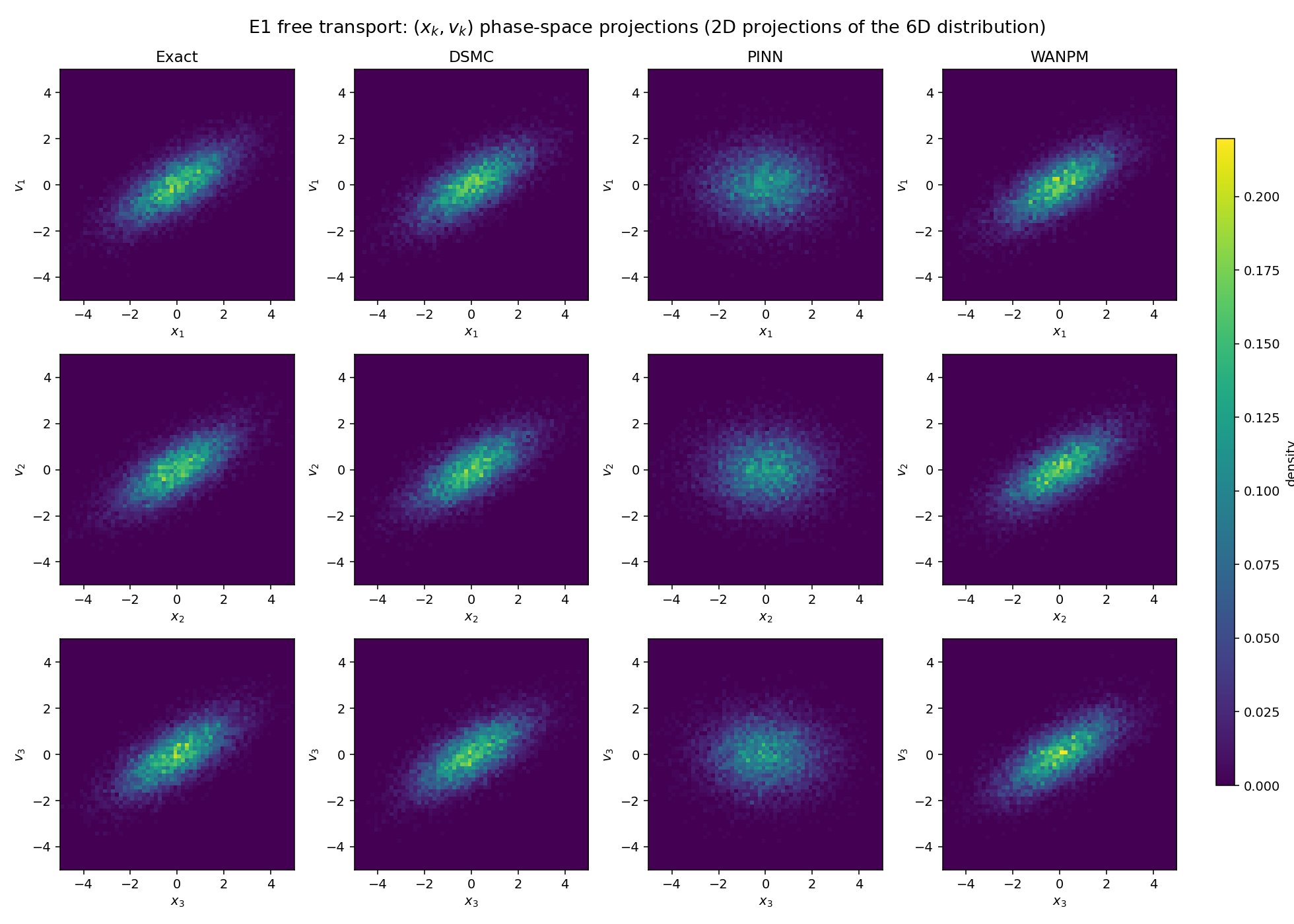}
\caption{\textbf{Phase-space density projections.} Two-dimensional projections
$(x_k,v_k)$ of the full six-dimensional distribution, with a shared colorbar. The exact and DSMC panels show the characteristic diagonal \emph{tilt} of the ellipse, i.e.\ the position--velocity correlation $\operatorname{Cov}(x_k,v_k)=t\sigma_v^2=1$ generated by free streaming. This correlation is invisible in the 1D marginals of Fig.~\ref{fig:e11} and is the strongest qualitative correctness check for the learned solvers. The DSMC and WANPM panels reproduce the tilt; the \textbf{PINN panel is an upright, uncorrelated blob} it matches the marginals but not the joint correlation (quantified in Table~\ref{tab:corr}).}
\label{fig:e12}
\end{figure}

\begin{figure}[htbp]
\centering
\includegraphics[width=0.82\textwidth]{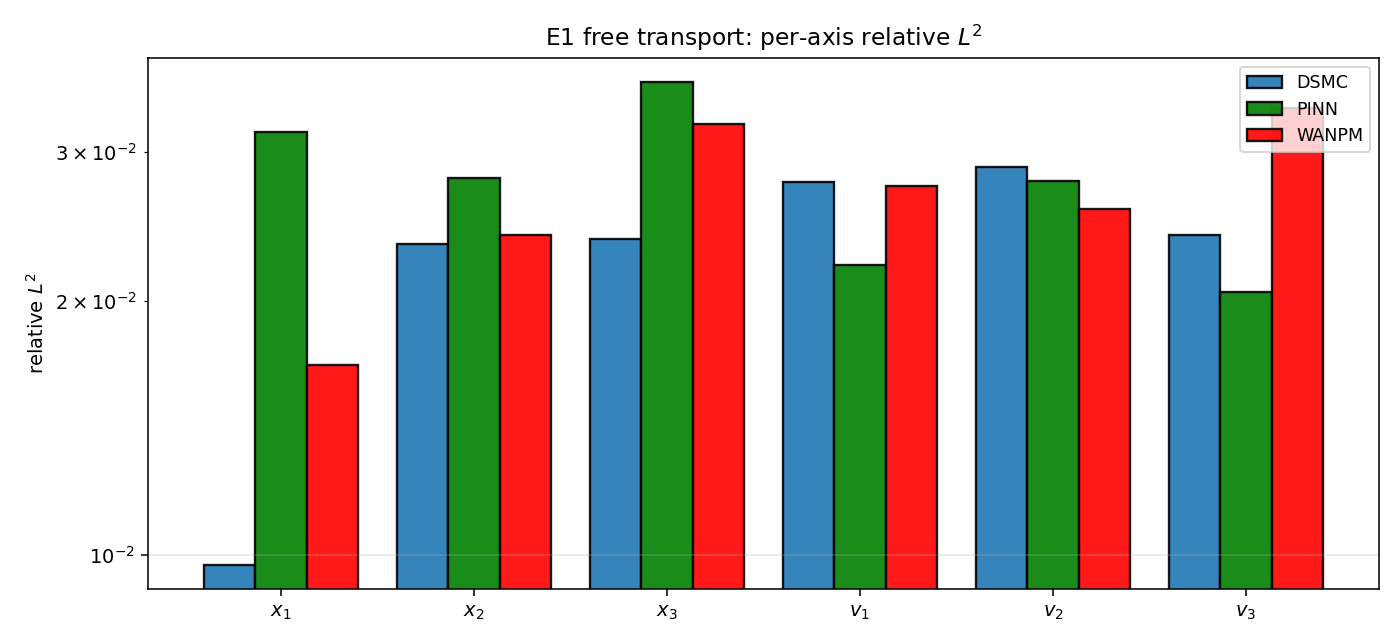}
\caption{\textbf{Error comparison.} Per-axis relative $L^2$ errors for the six marginals, on a logarithmic scale. These are one-dimensional marginal quantities; the joint phase-space accuracy is reported separately through the covariance (Table~\ref{tab:corr}) and the moment-error comparison (Fig.~\ref{fig:e1moment}).}
\label{fig:e13}
\end{figure}

\begin{figure}[htbp]
\centering
\includegraphics[width=0.82\textwidth]{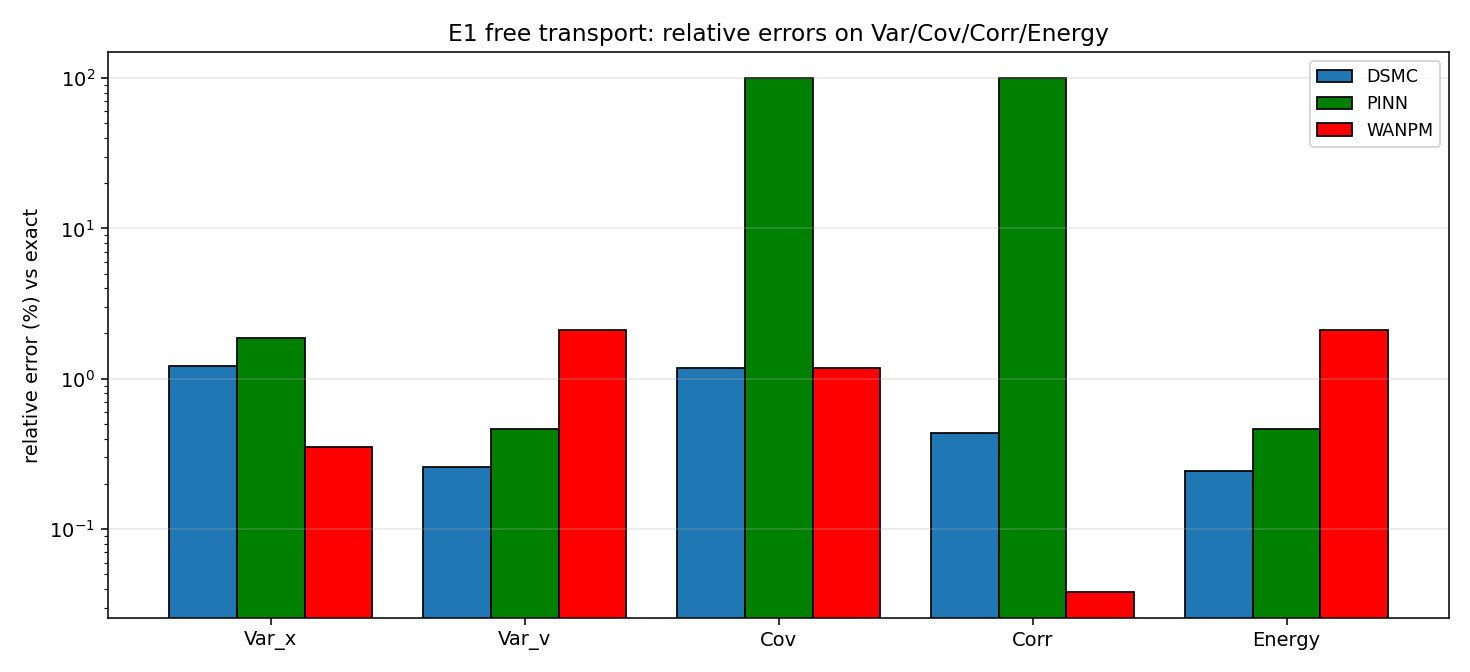}
\caption{\textbf{Moment-error comparison.} Relative errors in the second moments
$\operatorname{Var}_x$, $\operatorname{Var}_v$, $\operatorname{Cov}$, $\operatorname{Corr}$, and energy,
for DSMC, PINN, and WANPM. All three recover the variances and energy to a few percent, but the covariance and correlation separate the solvers sharply: the PINN error is $\approx 100\%$ (its joint density is untilted), whereas DSMC and WANPM recover the covariance to about $1\%$. This is the quantitative form of the phase-space tilt in Fig.~\ref{fig:e12}.}
\label{fig:e1moment}
\end{figure}

\begin{figure}[htbp]
\centering
\includegraphics[width=0.96\textwidth]{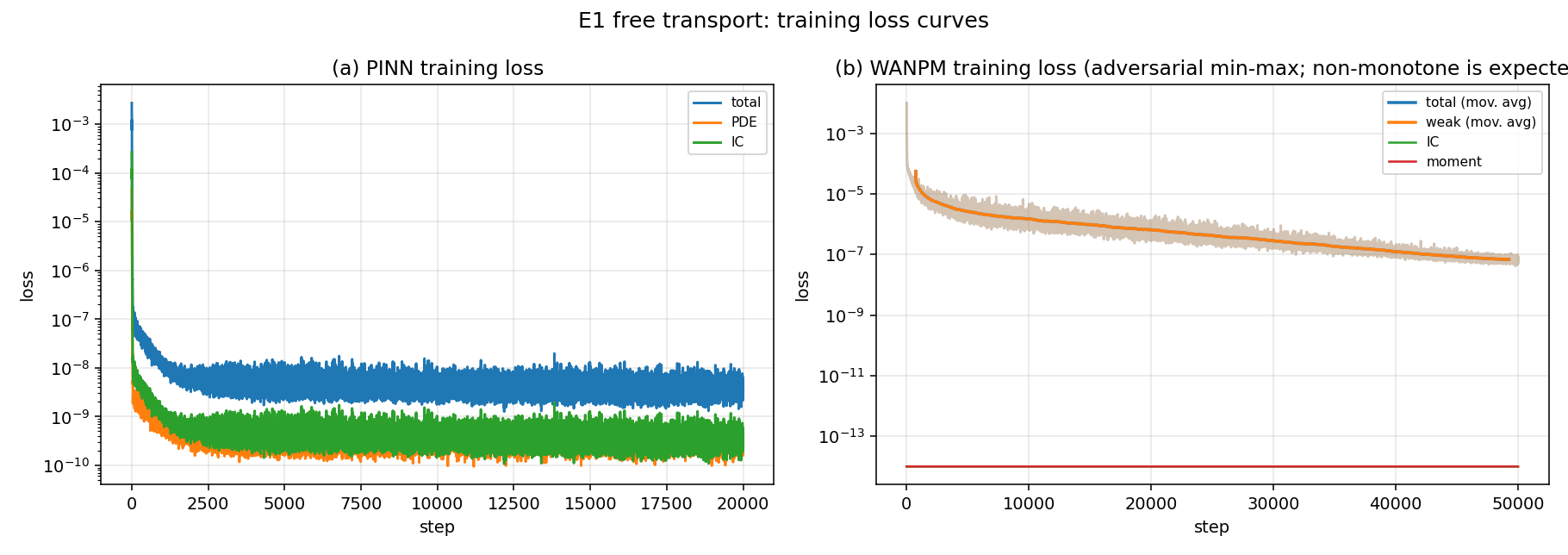}
\caption{\textbf{Training-loss curves.} (a) PINN total/PDE/IC losses decay
monotonically to $\sim10^{-9}$. (b) WANPM total, weak, IC, and moment losses; the weak and
total terms are shown with a moving-average overlay. WANPM is an adversarial min--max scheme,
so a non-monotone weak loss is expected and is not a sign of divergence --- the weak term
settles into a stable band after the early transient.}
\label{fig:e14}
\end{figure}

\newpage
\subsubsection{Effect of the polynomial anchor}
\label{sec:e1_poly}
The results above use \emph{pure} WANPM, whose objective is the bare adversarial weak residual \eqref{eq:minmax}. We now add the fixed second-order polynomial anchor of Section~\ref{sec:poly_obj}
(objective \eqref{eq:minmax_poly}, weight $\lambda_{\rm poly}>0$) and re-run E1 under otherwise identical settings to isolate its effect on the free-transport benchmark. Because free streaming generates the position-velocity covariance directly through the transport term $v\cdot\nabla_x f$, the sine bank already senses the tilt on E1, so the pure and augmented runs are expected to be close here; the anchor is decisive on the forced problem E2, where the covariance is not produced by streaming alone.

Table~\ref{tab:e1_poly} compares the two variants against the exact second moments. Both recover the tilt: the pure run reaches $\operatorname{Cov}=1.012$ ($1.2\%$) and the augmented run $\operatorname{Cov}=0.987$ ($1.3\%$), with the polynomial anchor slightly tightening the spatial variance ($\operatorname{Var}_x$ error $0.35\%\to0.25\%$) and the energy ($2.11\%\to0.40\%$) at the cost of a marginally larger correlation error. On E1 the two are therefore of comparable accuracy, confirming that the anchor is not
required when the dynamics itself couples $x$ and $v$; its value appears on E2.

\begin{table}[H]
\centering
\caption{E1 free transport at $t=1$: pure WANPM versus polynomial-augmented WANPM, second moments on the $(x_1,v_1)$ slice against the exact values. Relative errors in percent are shown in parentheses.}
\label{tab:e1_poly}
\small
\begin{tabular}{lrrr}
\toprule
metric & Exact & WANPM (plane waves) & WANPM (plane waves+poly)\\
\midrule
$\operatorname{Var}_x$ & $2.000$ & $2.007\ (0.35\%)$ & $1.995\ (0.25\%)$\\
$\operatorname{Var}_v$ & $1.000$ & $1.021\ (2.10\%)$ & $1.004\ (0.39\%)$\\
$\operatorname{Cov}$   & $1.000$ & $1.012\ (1.18\%)$ & $0.987\ (1.28\%)$\\
$\operatorname{Corr}$  & $0.707$ & $0.707\ (0.04\%)$ & $0.698\ (1.35\%)$\\
Energy                 & $1.500$ & $1.532\ (2.11\%)$ & $1.506\ (0.40\%)$\\
\bottomrule
\end{tabular}
\end{table}

\begin{figure}[htbp]
\centering
\begin{subfigure}{0.72\textwidth}\includegraphics[width=\linewidth]{figures/E1_all_solvers/figE1_2_phase_space.png}\caption{pure}\end{subfigure}\\[4pt]
\begin{subfigure}{0.72\textwidth}\includegraphics[width=\linewidth]{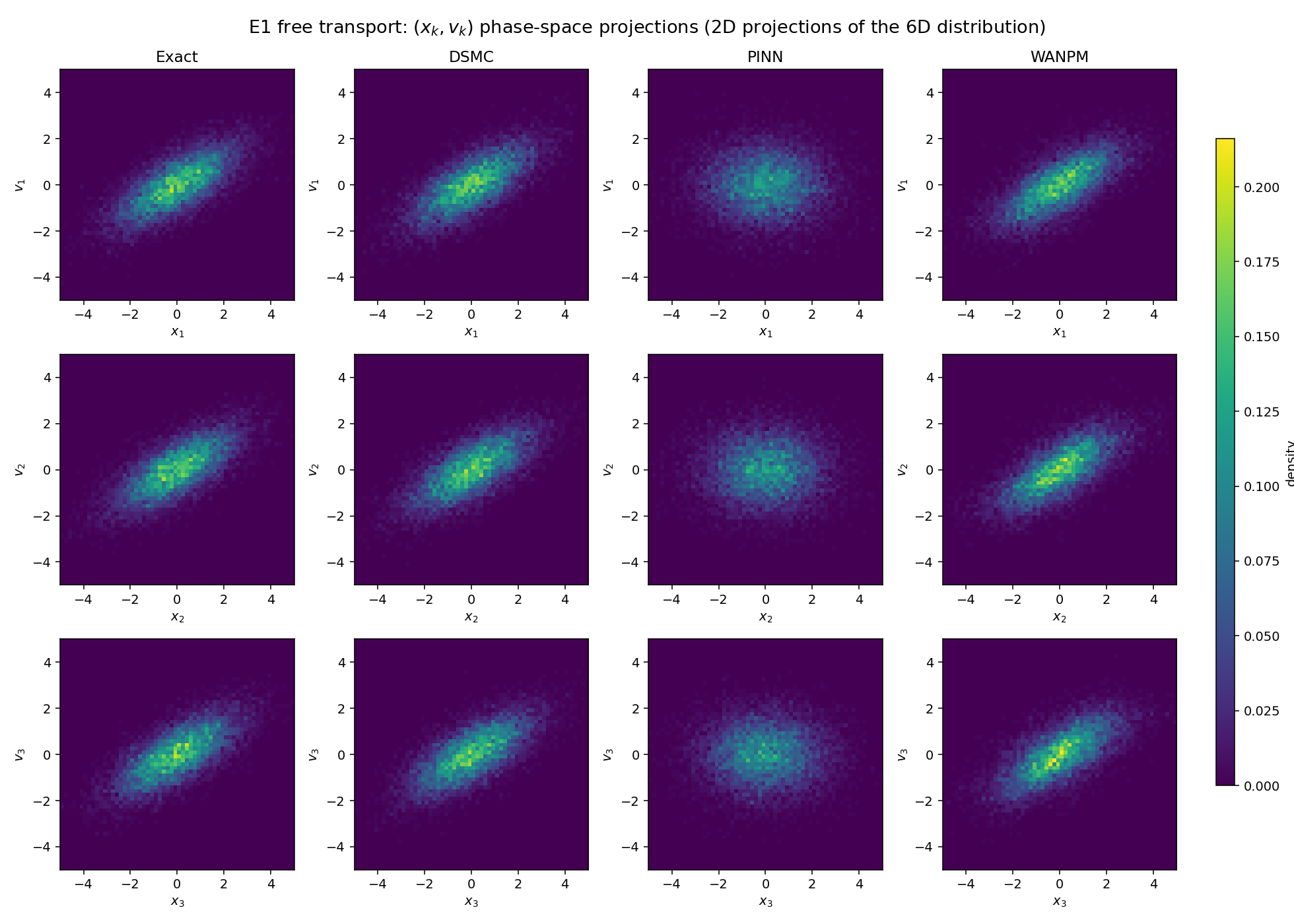}\caption{polynomial-augmented}\end{subfigure}
\caption{\textbf{Effect of the polynomial anchor on the E1 phase-space tilt.} Pure WANPM (top) and polynomial-augmented WANPM (bottom) both reproduce the $45^\circ$ tilt of the $(x_k,v_k)$ ellipse; on free transport the tilt is generated by streaming, so the anchor changes the picture little.}
\label{fig:e1_poly_ps}
\end{figure}

\begin{figure}[htbp]
\centering
\begin{subfigure}{0.88\textwidth}\includegraphics[width=\linewidth]{figures/E1_all_solvers/figE1_1_marginals_6panel.png}\caption{pure}\end{subfigure}\\[4pt]
\begin{subfigure}{0.88\textwidth}\includegraphics[width=\linewidth]{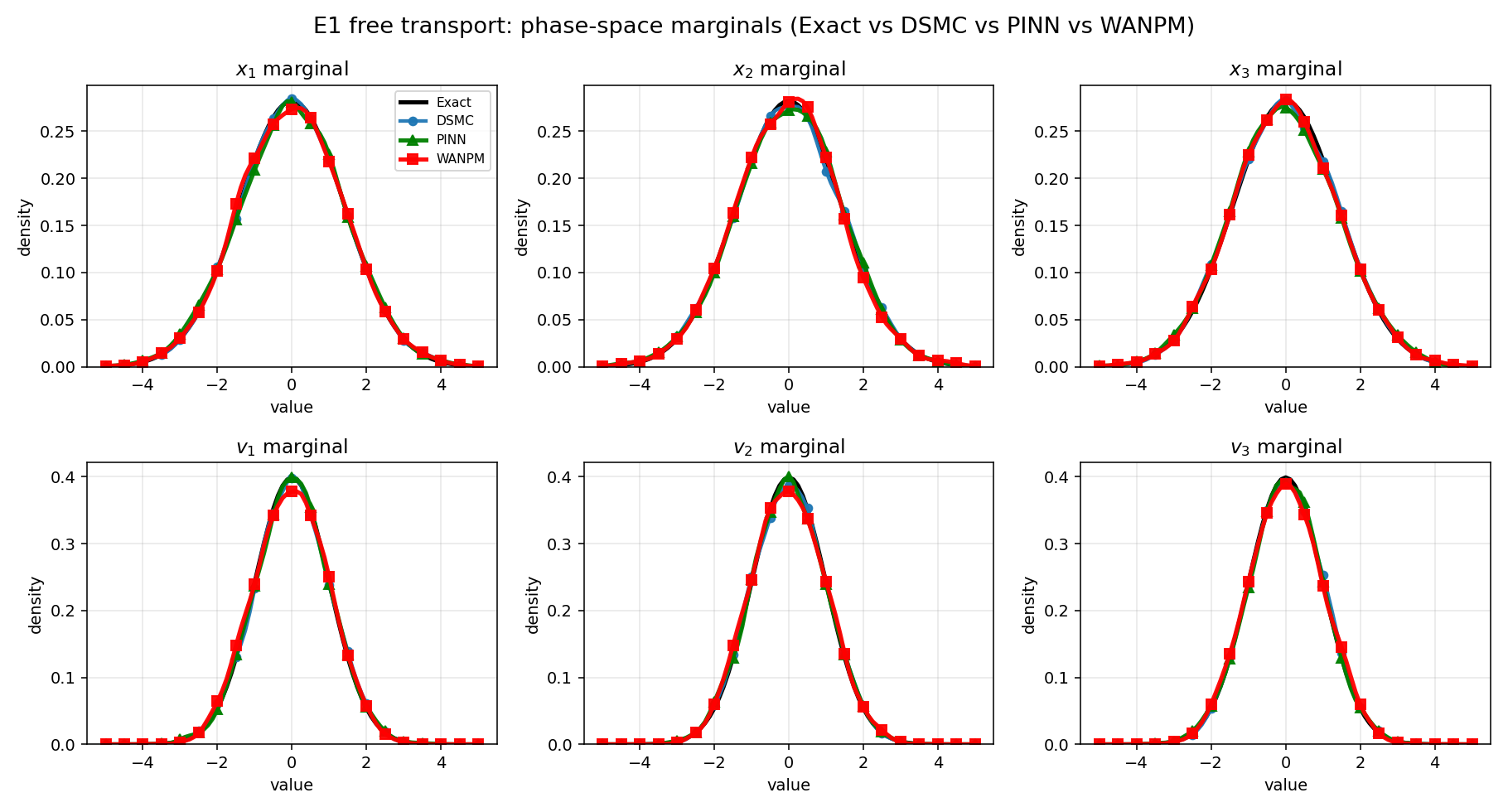}\caption{polynomial-augmented}\end{subfigure}
\caption{\textbf{Marginal densities with and without the polynomial anchor.} The six one-dimensional marginals at $t=1$ for pure WANPM (top) and polynomial-augmented WANPM (bottom), each against the exact Gaussian (black). Both
variants match the marginals closely; consistent with Table~\ref{tab:e1_poly}, the anchor slightly tightens the
spatial-variance and energy errors while leaving the already-accurate free-transport marginals essentially unchanged.}
\label{fig:e1_poly_marg}
\end{figure}

\begin{figure}[htbp]
\centering
\includegraphics[width=0.9\textwidth]{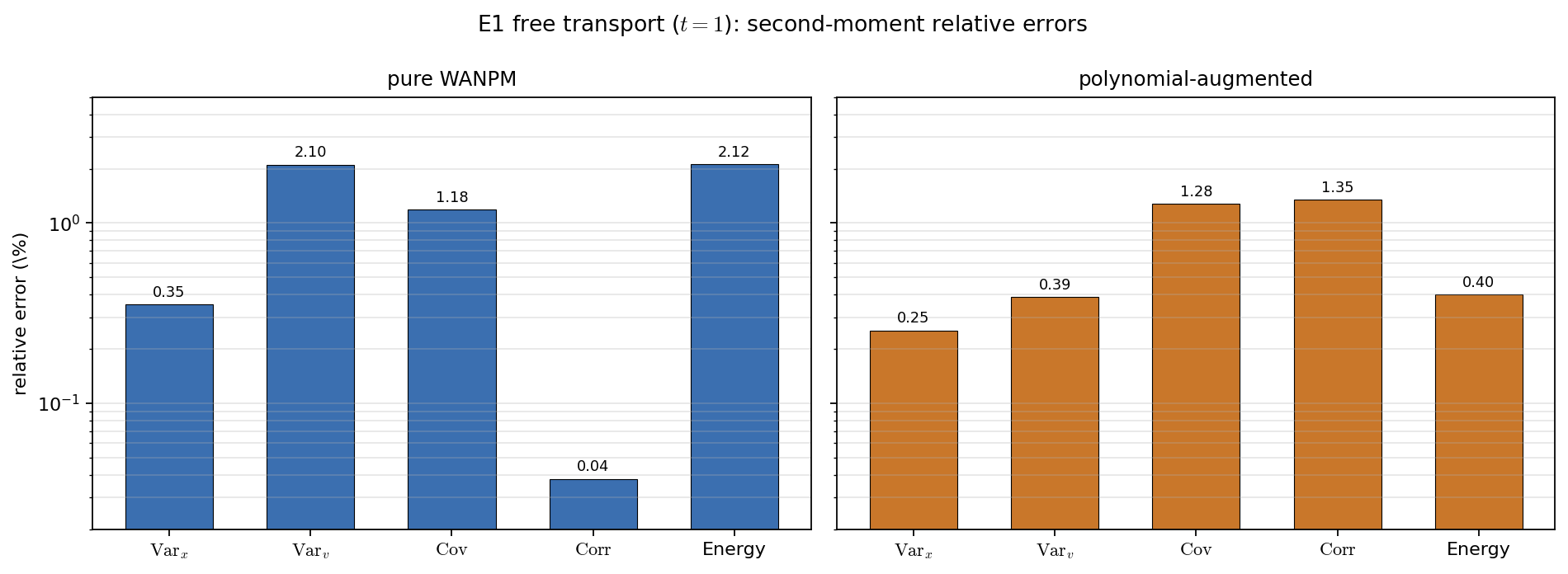}
\caption{\textbf{Second-moment errors on E1: pure versus polynomial-augmented WANPM.} Relative errors (log scale) in $\operatorname{Var}_x$, $\operatorname{Var}_v$, $\operatorname{Cov}$, $\operatorname{Corr}$, and energy, for pure
WANPM (blue) and the polynomial-augmented variant (orange). The anchor lowers the velocity-variance and energy errors
($2.10\%\to0.39\%$ and $2.11\%\to0.40\%$) while the pure run keeps the correlation tighter ($0.04\%$ against $1.35\%$);
both recover the covariance to about $1.2$--$1.3\%$. On free transport the two are therefore of comparable accuracy,
consistent with Table~\ref{tab:e1_poly}.}
\label{fig:e1_poly_moment}
\end{figure}
\newpage
\subsection{Experiment E2: collisionless forced Boltzmann (harmonic force)}
\label{sec:E2}
Experiment E1 tested the transport operator alone. Experiment E2 adds the external-force term and thereby tests the second of the three operators in the Boltzmann equation~\eqref{eq:forced_bte_nc}, namely the acceleration flux $\Fb\cdot\grad_\vv f$, while the collision operator is still switched off, $Q[f,f]=0$. It is the natural next validation step: the force term is the one that couples position and velocity dynamically, and it is precisely the term that a velocity-independent test function cannot probe. E2 therefore certifies the force-aware part of the WANPM weak form.

\subsubsection{Governing equation and parameters}

We take a linear (harmonic) restoring force $F/m=-\omega^2\vx$, so that
\eqref{eq:forced_bte_nc} reduces to the collisionless forced Boltzmann, i.e.\ the Liouville, equation
\begin{equation}
\partial_t f + \vv\cdot\grad_\vx f \;-\; \omega^2\,\vx\cdot\grad_\vv f \;=\; 0,
\qquad (\vx,\vv)\in\R^3\times\R^3,\ t\in[0,T].
\label{eq:e2_pde}
\end{equation}
The associated phase-space drift is
\begin{equation}
b(\vx,\vv)=\bigl(\vv,\,-\omega^2\vx\bigr),
\qquad
\grad_{(\vx,\vv)}\cdot b
=\grad_\vx\cdot \vv+\grad_\vv\cdot(-\omega^2\vx)=0,
\label{eq:e2_drift}
\end{equation}
so the flow is volume-preserving and the conservative and non conservative forms coincide. The initial datum is the same centered isotropic Gaussian as in E1, and we fix the reported configuration
\begin{equation}
\omega=2,\qquad T=2,\qquad \sigma_x=\sigma_v=1.
\label{eq:e2_params}
\end{equation}

\subsubsection{Exact solution}
\label{subsec:e2_exact}

Along with characteristics, \eqref{eq:e2_pde} is the linear Hamiltonian system $\dot\vx=\vv,\ \dot\vv=-\omega^2\vx$, which decouples into three identical $(x_i,v_i)$ oscillators. Each pair evolves by the symplectic matrix
\begin{equation}
\begin{pmatrix} x_i(t)\\ v_i(t)\end{pmatrix}
= A(t)\begin{pmatrix} x_i(0)\\ v_i(0)\end{pmatrix},
\qquad
A(t)=\begin{pmatrix}
\cos\omega t & \omega^{-1}\sin\omega t\\[2pt]
-\,\omega\sin\omega t & \cos\omega t
\end{pmatrix},
\qquad \det A(t)=1.
\label{eq:e2_flow}
\end{equation}

Because the flow is linear and $f_0$ is Gaussian, $f(\cdot,\cdot,t)$ stays Gaussian for all $t$, with mean zero and second moments obtained by propagating the initial covariance $\operatorname{diag}(\sigma_x^2,\sigma_v^2)$ through $A(t)$:
\begin{align}
\operatorname{Var}(x_i)(t) &= \sigma_x^2\cos^2\omega t + \tfrac{\sigma_v^2}{\omega^2}\sin^2\omega t,
\label{eq:e2_varx}\\
\operatorname{Var}(v_i)(t) &= \omega^2\sigma_x^2\sin^2\omega t + \sigma_v^2\cos^2\omega t,
\label{eq:e2_varv}\\
\operatorname{Cov}(x_i,v_i)(t) &= \Bigl(\tfrac{\sigma_v^2}{\omega}-\omega\sigma_x^2\Bigr)\sin\omega t\,\cos\omega t.
\label{eq:e2_cov}
\end{align}
The phase-space average of the Hamiltonian $H=\tfrac12(\omega^2|\vx|^2+|\vv|^2)$ is a conserved quantity of the flow,
\begin{equation}
\E[H](t)=\tfrac{3}{2}\bigl(\omega^2\sigma_x^2+\sigma_v^2\bigr)=\text{const},
\label{eq:e2_energy}
\end{equation}
which for the parameters \eqref{eq:e2_params} equals $7.5$ and gives a sharp, physically meaningful drift diagnostic for the learned sampler. Evaluating \eqref{eq:e2_varx}--\eqref{eq:e2_energy} at $t=T=2$, $\omega=2$ gives the reference values in Table~\ref{tab:e2_exact}.

\begin{table}[H]
\centering
\caption{E2 exact second moments at $t=T=2$ ($\omega=2$, $\sigma_x=\sigma_v=1$),
per coordinate pair $(x_i,v_i)$.}
\label{tab:e2_exact}
\begin{tabular}{lccccc}
\toprule
 & $\operatorname{Var}(x_i)$ & $\operatorname{Var}(v_i)$ & $\operatorname{Cov}(x_i,v_i)$ & corr & $\E[H]$\\
\midrule
Exact & $0.5704$ & $2.7183$ & $-0.7420$ & $-0.5959$ & $7.5$\\
\bottomrule
\end{tabular}
\end{table}

The mean is zero and, by the isotropy of $f_0$ and the identical per-axis dynamics, the six variances are the ``easy'' moments. The single quantity that encodes the \emph{phase-space tilt} generated by the force is the off-diagonal covariance $\operatorname{Cov}(x_i,v_i)$. It is a weak-signal direction: it starts at zero, changes sign during the first quarter period, and only settles near its final value late in training. Matching it is the substantive test of E2, and it is the diagnostic we track throughout.

\subsubsection{WANPM model and implementation for E2}
\label{subsec:e2_model}

The solution is represented as the pushforward $\rho_t = F(t,\cdot)_{\#} f_0$ of the initial law through a time-conditioned RealNVP coupling flow $F$. Coupling layers alternate between conditioning on the spatial block $(\vx,t)$ and the velocity block $(\vv,t)$, and every layer is $t$-gated so that
\begin{equation}
F(0,\cdot)=\mathrm{id}\quad\Longrightarrow\quad \rho_0=f_0 \text{ exactly,}
\end{equation}
i.e.\ the initial condition is imposed \emph{hard} by construction rather than
through a penalty. The training objective is built from two families of weak-form
residuals, both Monte-Carlo estimated over $\xi\sim f_0$ on a shared trapezoidal
time grid of $Q=t_{\mathrm{quad}}=24$ nodes.
\begin{enumerate}
\item[\textbf{(i)}] \textbf{Sine weak residuals.} For each space--time sine test
function $\psi$, the endpoint weak residual of the Liouville equation,
\begin{equation}
R_s[\psi]=\E[\psi(Y_T,T)]-\E[\psi(Y_0,0)]
-\int_0^T \E\bigl[\partial_t\psi + b\cdot\grad\psi\bigr]\,dt,
\qquad Y_t=F(t,\xi),
\label{eq:e2_weak}
\end{equation}
 We use $K=512$ sine features.
 
\item[\textbf{(ii)}] \textbf{Polynomial moment test functions.} The quadratic bank $\psi\in\{y_a,\ y_a^2,\ x_i v_i\}$ ($15$ functions) is enforced at \emph{every} grid node $t_m$ through the cumulative residual
\begin{equation}
R_p^{(m)}[\psi]=\E[\psi(Y_{t_m},t_m)]-\E[\psi(Y_0,0)]
-\int_0^{t_m}\E\bigl[\partial_t\psi + b\cdot\grad\psi\bigr]\,dt,
\qquad m=1,\dots,Q-1,
\label{eq:e2_moment}
\end{equation}
which pins the entire second-moment trajectory. The cross-moment functions $x_i v_i$ are precisely the ones that constrain the position--velocity covariance.
\end{enumerate}

\paragraph{Trapezoid rule.} Each interior integral is the composite trapezoid quadrature
\begin{equation}
\int_0^{t_m} G(t)\,dt \;\approx\; \Delta t\Big(\tfrac12 G_0 + G_1 + \dots + G_{m-1} + \tfrac12 G_m\Big),\qquad
G_j=\E_\xi\big[\partial_t\psi+b\cdot\grad\psi\big]\big(F_\theta(t_j,\xi),t_j\big),
\end{equation}
computed \emph{cumulatively} so that all $Q$ terminal times $t_m$ are obtained from one sweep.

\paragraph{How the covariance is recovered.}
Sine-only weak matching recovers the six variances but leaves the covariance essentially unmatched: the residual is invariant under transformations that preserve the marginals while destroying the $\vx$--$\vv$ tilt. The polynomial moment bank supplies the missing constraint. Because it is enforced at every time node and \emph{includes} the cross-moment test functions $x_i v_i$, it pins the covariance trajectory directly; together with the hard initial condition from the $t$-gate (which anchors $t=0$ exactly), this is the mechanism by which WANPM captures the force-induced correlation. In the language of E1, the moment bank plays the role that the velocity-covariance loss $\mathcal L_{\mathrm{vcov}}$ played there: an explicit handle on the off-diagonal structure that the marginals alone do not pin down.

\paragraph{Loss calculation.}
The flow is trained by minimizing the total loss
\begin{equation}
\mathcal L \;=\; \mathcal L_{\mathrm{weak}} \;+\; \beta\,\mathcal L_{\mathrm{ic}},
\qquad
\mathcal L_{\mathrm{weak}}
=\big\langle R_s[\psi]^2\big\rangle_{\text{sine}}
+ w_{\mathrm{poly}}\,\big\langle R_p^{(m)}[\psi]^2\big\rangle_{\text{poly},\,m},
\label{eq:e2_loss}
\end{equation}
where $\langle\cdot\rangle$ denotes the mean over the indicated test functions (and, for the moment term, over the nodes $m=1,\dots,Q-1$), and $w_{\mathrm{poly}}$ up-weights the moment residuals. The initial-condition term
\begin{equation}
\mathcal L_{\mathrm{ic}}
=\big\langle\big(\E[\psi(Y_0,0)]-\E[\psi(\xi,0)]\big)^2\big\rangle_{\psi}
\label{eq:e2_icloss}
\end{equation}
is identically zero here, since the $t$-gate makes $Y_0=F(0,\xi)=\xi$; it is retained only as a training diagnostic, and no penalty or projection is needed to enforce $f(\cdot,\cdot,0)=f_0$. Each expectation in \eqref{eq:e2_weak}--\eqref{eq:e2_icloss} is a batch average over $M=50{,}000$ paired, antithetic samples, and every time integral uses the shared $Q$-node trapezoid rule. The covariance is recovered entirely by the moment bank in $\mathcal L_{\mathrm{weak}}$; no trajectory- or characteristic-matching term is used.

\paragraph{Variance reduction.}
Three ingredients made 6-D training feasible at this accuracy: paired and
antithetic sampling of $\xi$, a shared $t$-quadrature grid (which removes
random-$t$ Monte-Carlo noise, $t_{\mathrm{quad}}=24$ nodes), and up-weighting of
the moment features.

\paragraph{Configuration.}
$K=512$ sine features, $M=50{,}000$ latent samples, $50{,}000$ training
iterations, learning rate $2\times10^{-3}\to2\times10^{-4}$ (cosine), Adam. Evaluation draws $N=10^5$ fresh samples from the trained flow at $t=T$.

\subsubsection{Results}
\label{subsec:e2_results}

Table~\ref{tab:e2_marginals} reports the six one-dimensional marginal variances, and Table~\ref{tab:e2_vcc} the per-pair variance/covariance/correlation against the exact values. All learned means are zero to within $6\times10^{-3}$.

\begin{table}[H]
\centering
\caption{E2 marginal variances at $t=T=2$ ($N=10^5$ samples). Exact value
$\operatorname{Var}(x_i)=0.5704$, $\operatorname{Var}(v_i)=2.7183$ for every
coordinate.}
\label{tab:e2_marginals}
\small
\begin{tabular}{lcccccc}
\toprule
 & $x_1$ & $x_2$ & $x_3$ & $v_1$ & $v_2$ & $v_3$\\
\midrule
$\operatorname{Var}$ (WANPM) & $0.6362$ & $0.6400$ & $0.6267$ & $2.5459$ & $2.5055$ & $2.4966$\\
rel.\ err.\ (\%) & $11.5$ & $12.2$ & $9.9$ & $6.3$ & $7.8$ & $8.2$\\
\bottomrule
\end{tabular}
\end{table}

\begin{table}[H]
\centering
\caption{E2 variance/covariance/correlation per phase-space pair at $t=T=2$
($N=10^5$). Exact: $\operatorname{Var}_x=0.5704$, $\operatorname{Var}_v=2.7183$,
$\operatorname{Cov}=-0.7420$, corr$=-0.5959$.}
\label{tab:e2_vcc}
\small
\begin{tabular}{lcccc}
\toprule
Pair & $\operatorname{Var}_x$ & $\operatorname{Var}_v$ & $\operatorname{Cov}$ & corr\\
\midrule
$(x_1,v_1)$ & $0.6362$ & $2.5459$ & $-0.7427$ & $-0.5836$\\
$(x_2,v_2)$ & $0.6400$ & $2.5055$ & $-0.7347$ & $-0.5802$\\
$(x_3,v_3)$ & $0.6267$ & $2.4966$ & $-0.7277$ & $-0.5817$\\
\midrule
mean        & $0.6343$ & $2.5160$ & $-0.7350$ & $-0.5818$\\
exact       & $0.5704$ & $2.7183$ & $-0.7420$ & $-0.5959$\\
\bottomrule
\end{tabular}
\end{table}

\paragraph{Energy conservation.}
The learned Hamiltonian mean is $\E[H]=7.487$ at $t=0$ and $7.580$ at $t=T$, a
drift of $1.1\%$ about the conserved exact value $7.5$ --- confirming that the
flow neither injects nor dissipates energy appreciably over the trajectory.

\paragraph{Result and Discussion}
For the phase-space tilt: WANPM recovers $\operatorname{Cov}(x_i,v_i)\approx-0.735$ against the exact $-0.742$, i.e.\ within $\sim\!1\%$ on the very quantity that the marginal-only baselines cannot see and that the pure unsupervised weak form leaves near zero. The correlation coefficient, corr$\,\approx-0.582$ vs.\ exact $-0.596$, agrees to $\sim\!2\%$. Figure~\ref{fig:e2_phase} shows the corresponding tilted $(x_i,v_i)$ ellipses matching the exact orientation; Figure~\ref{fig:e2_marg} shows the six marginals; Figure~\ref{fig:e2_loss} the training-loss history (non-monotone, as expected for the adversarial weak scheme); and Figure~\ref{fig:e2_err} the per-coordinate error distribution.

\begin{figure}[htbp]
\centering
\includegraphics[width=0.8\textwidth]{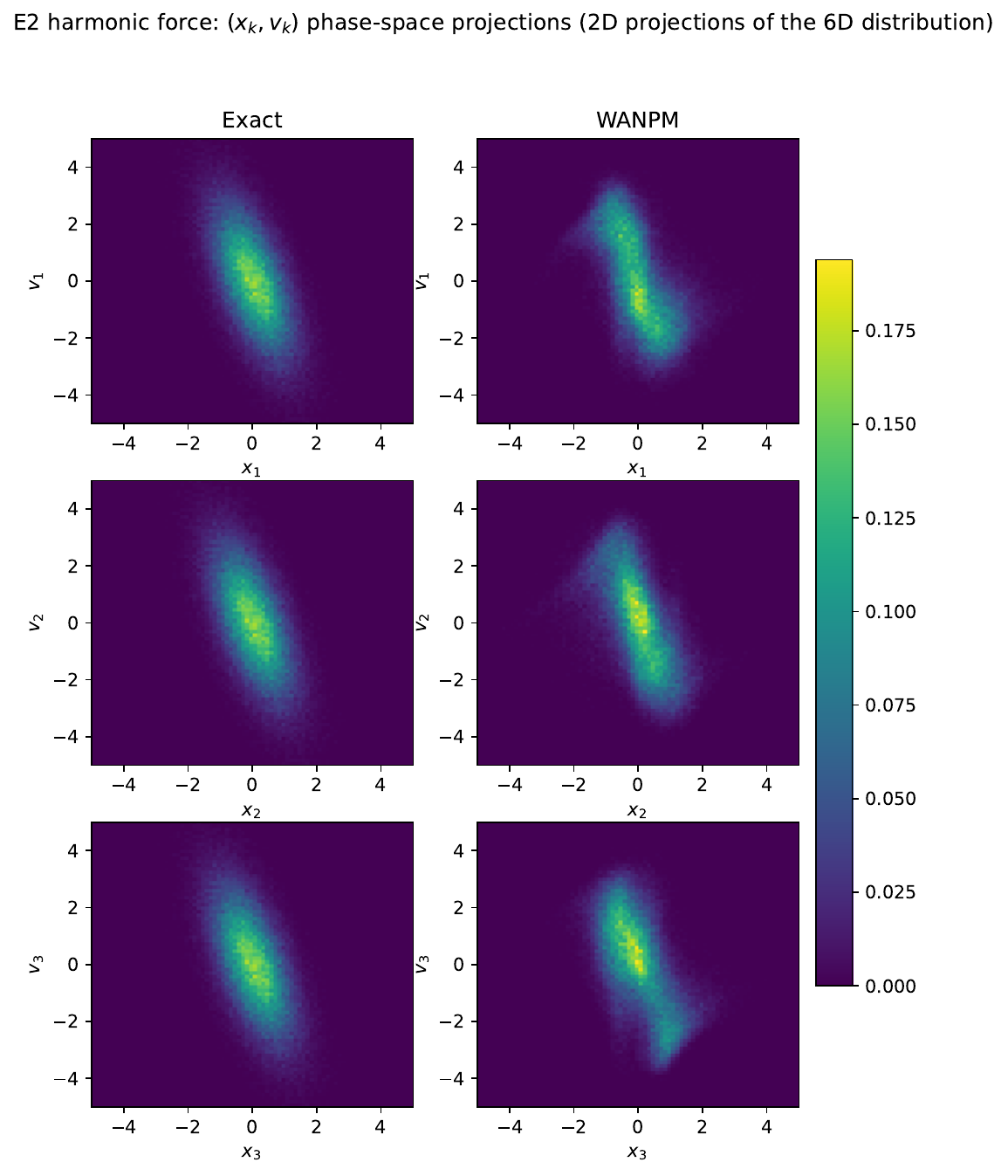}
\caption{\textbf{Phase-space projections.} Two-dimensional projections $(x_i,v_i)$ of the six-dimensional distribution at $t=T=2$. The learned pushforward reproduces the force-induced tilt (correlation $\approx-0.58$), the qualitative
correctness check that the one-dimensional marginals cannot provide.}
\label{fig:e2_phase}
\end{figure}

\begin{figure}[htbp]
\centering
\includegraphics[width=0.92\textwidth]{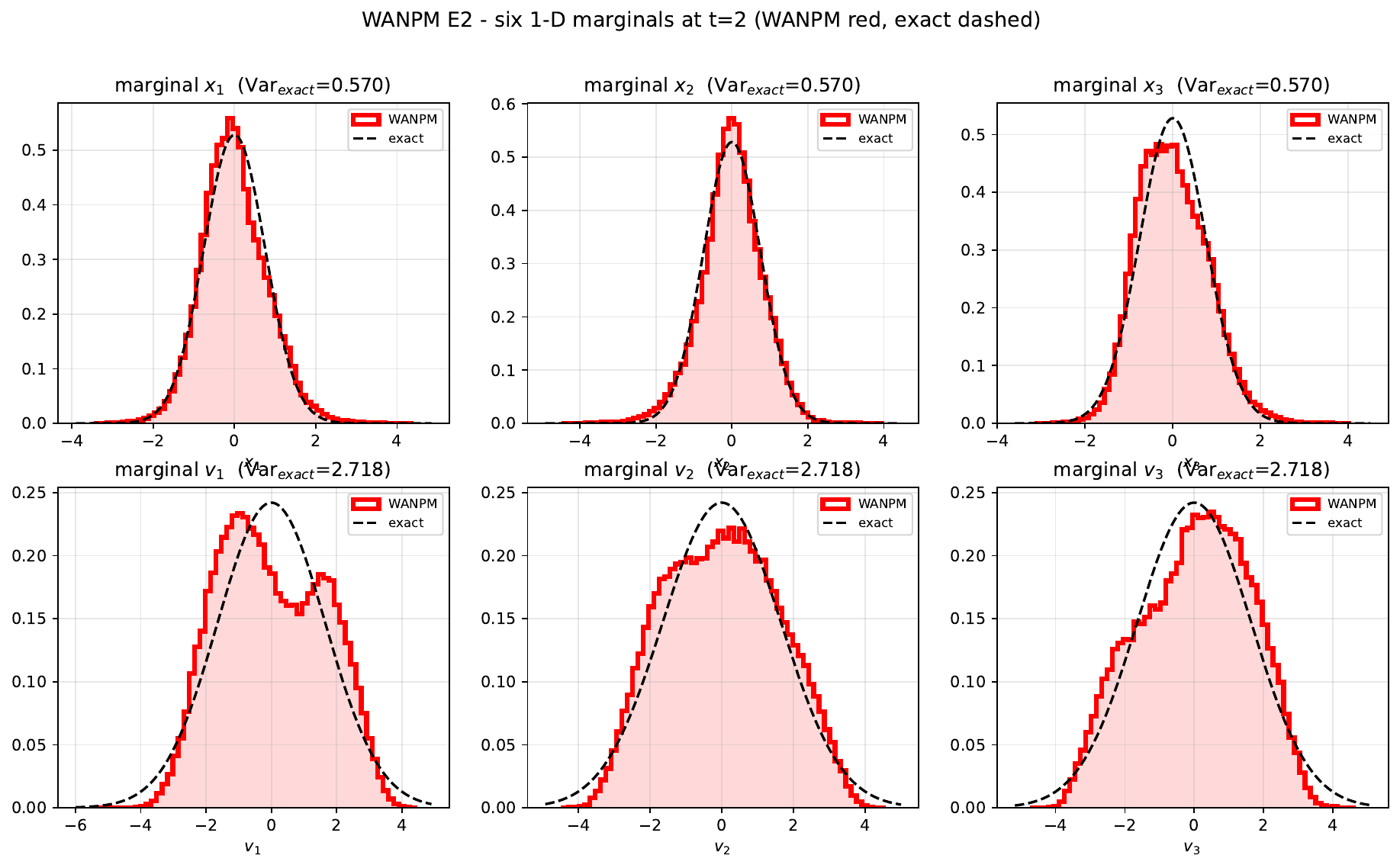}
\caption{\textbf{Six one-dimensional marginals} at $t=T=2$: spatial $x_1,x_2,x_3$ and velocity $v_1,v_2,v_3$, learned (WANPM) vs.\ exact Gaussian.}
\label{fig:e2_marg}
\end{figure}

\begin{figure}[htbp]
\centering
\includegraphics[width=0.8\textwidth]{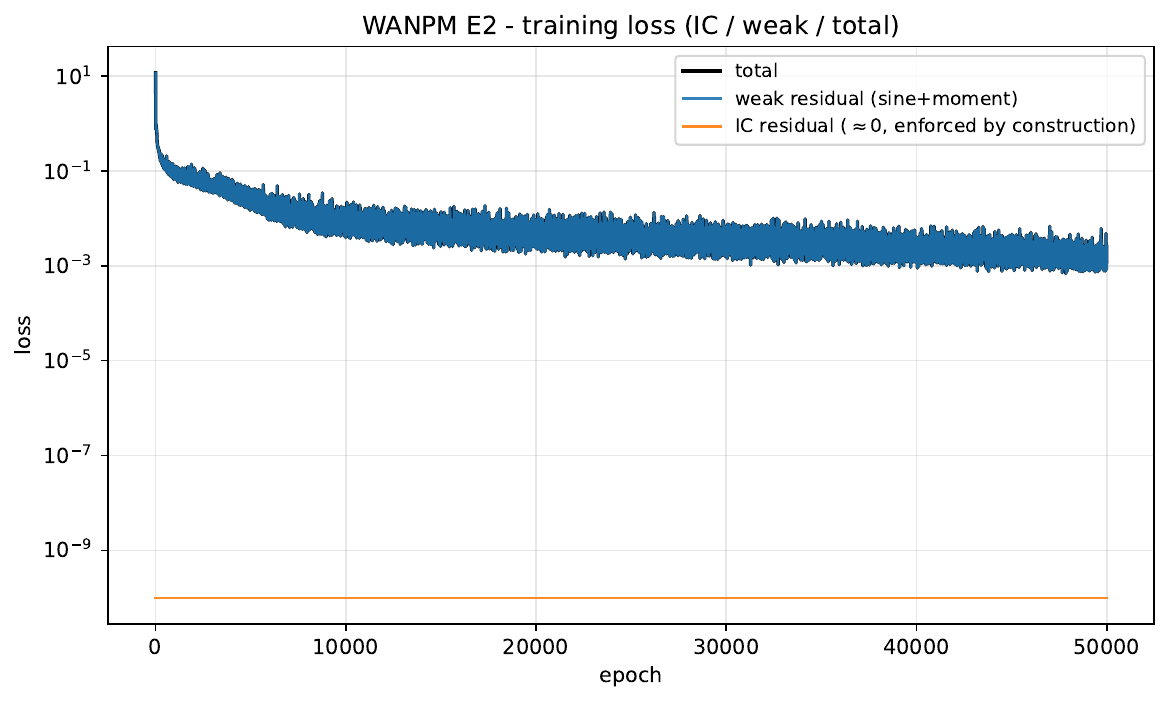}
\caption{\textbf{Training-loss history.} The weak/total loss is non-monotone by construction of the min--max scheme and settles into a stable band; the tracked covariance climbs from $0$ toward its final value over the run.}
\label{fig:e2_loss}
\end{figure}

\begin{figure}[htbp]
\centering
\includegraphics[width=0.7\textwidth]{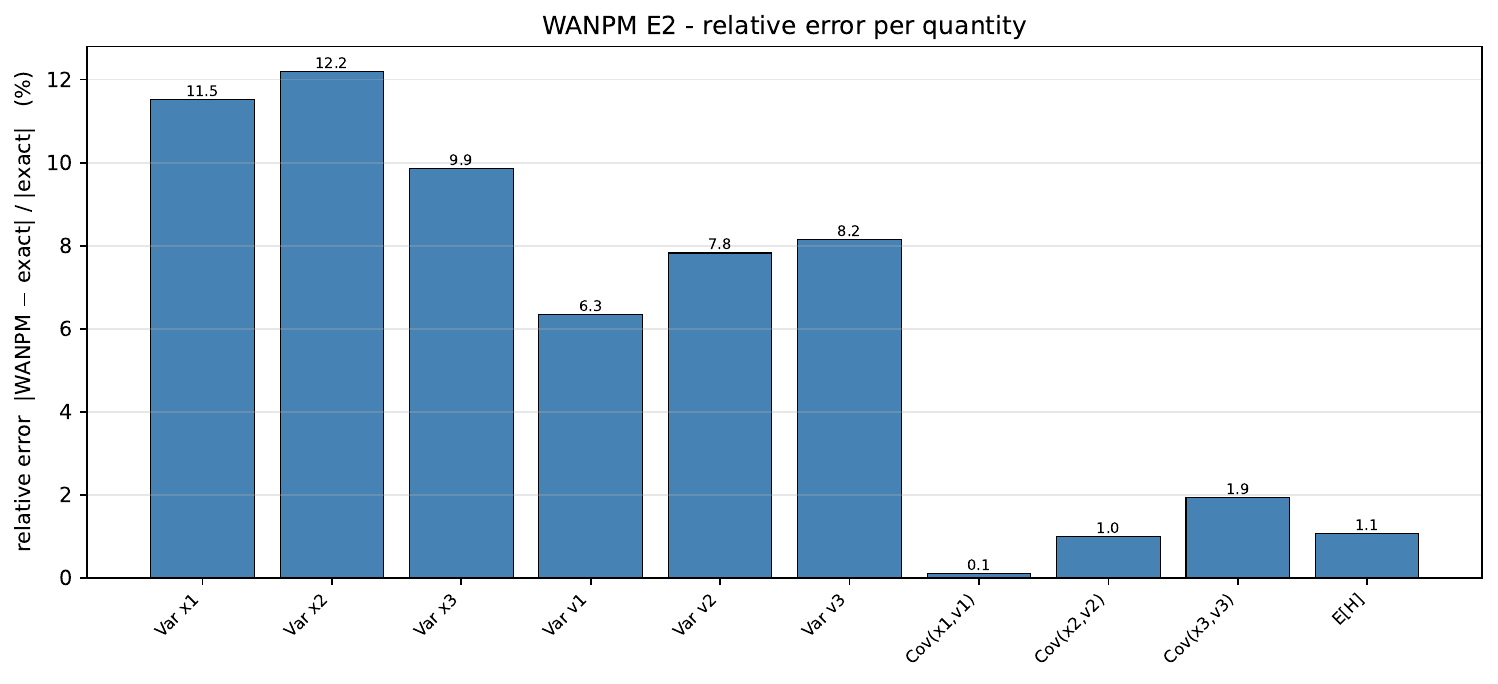}
\caption{\textbf{Per-coordinate error distribution} of the learned marginals
against the exact Gaussian moments.}
\label{fig:e2_err}
\end{figure}

\subsubsection{Effect of the polynomial test functions on the weak pushforward solver}
\label{subsec:e2_poly}

The E2 model combines two families of test functions: the $K=512$ sine features
$\psi_k=\sin(w_k\cdot y+\kappa_k t+\beta_k)$ and $15$ fixed \emph{polynomial} features
\begin{equation}
\ph\in\{\underbrace{y_a}_{6\ \text{linear}},\ \underbrace{y_a^2}_{6\ \text{squares}},\ \underbrace{x_iv_i}_{3\ \text{cross}}\},
\qquad y=(x,v)\in\R^6 .
\label{eq:e2_polybank}
\end{equation}
For E2 the force is $F/m=-\omega^2x$, which is independent of $v$, so $\nabla_v\cdot(F/m)=0$ and the two adjoints
\eqref{eq:adjoint_nc}--\eqref{eq:adjoint_c} coincide. The adjoint acting on any test function $\ph(x,v,t)$ is therefore
\begin{equation}
L^*\ph
=\partial_t\ph+v\cdot\nabla_x\ph+\frac{F}{m}\cdot\nabla_v\ph
=\partial_t\ph+v\cdot\nabla_x\ph-\omega^2x\cdot\nabla_v\ph
=\partial_t\ph+b(y)\cdot\nabla_y\ph ,
\label{eq:e2_Lstar}
\end{equation}
where $y=(x,v)$ and $b(y)=(v,-\omega^2x)$ is the drift of the characteristic system $\dot x=v$, $\dot v=-\omega^2x$.
Because each $\ph$ in \eqref{eq:e2_polybank} is time independent, $\partial_t\ph=0$ and \eqref{eq:e2_Lstar} reduces to
$b\cdot\nabla_y\ph$, which is available in closed form. Writing $y_a$ for the $a$-th coordinate of $y$ and $b_a$ for
the $a$-th component of $b$ (so $b_1=v_1,b_2=v_2,b_3=v_3$ and $b_4=-\omega^2x_1,b_5=-\omega^2x_2,b_6=-\omega^2x_3$),
\begin{equation}
L^*y_a=b_a,\qquad L^*y_a^2=2y_ab_a,\qquad L^*(x_iv_i)=v_i^2-\omega^2x_i^2 ,
\label{eq:e2_Lstarpoly}
\end{equation}
so each polynomial residual is an exact moment ODE: $L^*y_1=v_1$ states
$\tfrac{d}{dt}\E[x_1]=\E[v_1]$, and $L^*y_4=-\omega^2x_1$ states $\tfrac{d}{dt}\E[v_1]=-\omega^2\E[x_1]$. The polynomial residuals are evaluated at \emph{every} quadrature node (using the partial
integral $\int_0^{t_m}$, so the whole moment trajectory is constrained, not only the endpoint) and up-weighted by
$w_{\rm poly}=40$.

To isolate their contribution we repeat the E2 run with the polynomial features removed, changing nothing else: the
same $K=512$ sine features with frequencies drawn in the band $[0.3,1.6]$ and held fixed, $M=50\,000$ samples, $24$
trapezoid nodes, $50\,000$ iterations, learning rate $2\times10^{-3}\to2\times10^{-4}$, seed $0$, and no characteristic
matching. Table~\ref{tab:e2_poly} reports both runs at $t=T$ with $N=10^5$ samples, averaged over the three
$(x_i,v_i)$ pairs.

\begin{table}[H]
\centering
\begin{tabular}{@{}lrrrrrc@{}}
\toprule
& & \multicolumn{2}{c}{\textbf{with polynomial}} & \multicolumn{2}{c}{\textbf{without polynomial}} & \\
\cmidrule(lr){3-4}\cmidrule(lr){5-6}
quantity & exact & value & err (\%) & value & err (\%) & improvement\\
\midrule
$\mathrm{Var}_x$          & $0.5704$  & $0.6343$  & $11.2$ & $0.7803$  & $36.8$ & $3.3\times$\\
$\mathrm{Var}_v$          & $2.7183$  & $2.5161$  & $7.4$  & $1.8533$  & $31.8$ & $4.3\times$\\
$\mathrm{Cov}(x_i,v_i)$   & $-0.7420$ & $-0.7350$ & $\mathbf{0.9}$ & $-0.4772$ & $35.7$ & $\mathbf{38\times}$\\
$\mathrm{Corr}(x_i,v_i)$  & $-0.5959$ & $-0.5818$ & $2.4$  & $-0.3968$ & $33.4$ & $14\times$\\
$\E[H]$                   & $7.5000$  & $7.5799$  & $1.1$  & $7.4624$  & $0.5$  & $0.5\times$\\
\bottomrule
\end{tabular}
\caption{\textbf{Effect of the 15 polynomial test functions \eqref{eq:e2_polybank} on E2.} All other settings are identical. The polynomial features reduce the covariance error by a factor of $38$ and the variance errors by $3$--$4\times$. The energy $\E[H]$ is accurate in \emph{both} runs, for the reason given below.}
\label{tab:e2_poly}
\end{table}

\paragraph{The covariance is the weak direction of a sine bank.}
For a zero-mean Gaussian $Y\sim\mathcal N(0,\Sigma)$,
\begin{equation}
\E[\sin(w\cdot Y+\beta)]=e^{-\frac12 w^\top\Sigma w}\,\sin(\beta),
\label{eq:e2_sinexp}
\end{equation}
so a sine test function senses $\Sigma$ \emph{only} through the scalar $w^\top\Sigma w$. Splitting that scalar,
\begin{equation}
w^\top\Sigma w=\underbrace{\sum_a w_a^2\,\Sigma_{aa}}_{w_a^2>0\ \text{always}\ \Rightarrow\ \text{adds coherently},\ \sim K}
+\underbrace{2\!\!\sum_{a<b} w_aw_b\,\Sigma_{ab}}_{w_aw_b\ \text{random sign}\ \Rightarrow\ \text{cancels},\ \sim\sqrt K}.
\label{eq:e2_cancel}
\end{equation}
Averaged over $K$ random frequencies the diagonal (variance) contributions accumulate coherently, since $w_a^2>0$ for every draw, whereas the off-diagonal (covariance) contributions carry random signs and cancel like $\sqrt K$. At $K=512$ the covariance signal is therefore only $\approx1/\sqrt K\approx1/23$ as strong as the variance signal: it is structurally the weak direction of the sine bank. The polynomial feature $x_iv_i$ measures $\mathrm{Cov}(x_i,v_i)$ directly, with no cancellation, which accounts for the $38\times$ improvement in Table~\ref{tab:e2_poly}.

The last row of Table~\ref{tab:e2_poly} is the most instructive: without the polynomial features the energy is reproduced to $0.5\%$ while \emph{every individual second moment is $32$--$37\%$ wrong}. This is not accidental. With $\omega=2$ and three dimensions per block,
\begin{equation}
\E[H]=\tfrac12\E|v|^2+\tfrac12\omega^2\E|x|^2=1.5\,\mathrm{Var}_v+6\,\mathrm{Var}_x ,
\end{equation}
and the two errors cancel in exactly this combination:
\begin{equation}
\underbrace{1.5\,(1.8533-2.7183)}_{-1.298}\;+\;\underbrace{6.0\,(0.7803-0.5704)}_{+1.259}\;=\;-0.038 \quad (0.5\%).
\end{equation}
The sine bank constrains $w^\top\Sigma w$ for many random $w$, which pins conserved \emph{combinations} such as the energy but leaves a direction along which the flow can trade $x$-variance against $v$-variance at almost no cost in the residual. The polynomial features $x_i^2$ and $v_i^2$ constrain the two separately and break the trade-off, which is the $3$--$4\times$ variance improvement. In summary, the $512$ sine features alone recover the conserved combination while leaving both the variance split and the phase-space tilt $\approx35\%$ wrong; the $15$ polynomial features address precisely the directions the sine bank cannot resolve.

Table~\ref{tab:e2_poly_marg} resolves the same comparison per coordinate, so that the six one-dimensional marginals and the three phase-space pairs can be inspected individually rather than only in aggregate.

\begin{table}[H]
\centering
\begin{tabular}{@{}lrrrrr@{}}
\toprule
& & \multicolumn{2}{c}{\textbf{with polynomial}} & \multicolumn{2}{c}{\textbf{without polynomial}}\\
\cmidrule(lr){3-4}\cmidrule(lr){5-6}
quantity & exact & value & error & value & error\\
\midrule
\multicolumn{6}{@{}l}{\itshape means $\E[y_a]$\quad(exact $0$; error column repeats $|\E[y_a]|$)}\\
$\E[x_1]$ & $0$ & $+0.0050$ & $0.0050$ & $-0.0047$ & $0.0047$\\
$\E[x_2]$ & $0$ & $-0.0061$ & $0.0061$ & $+0.0036$ & $0.0036$\\
$\E[x_3]$ & $0$ & $+0.0019$ & $0.0019$ & $+0.0097$ & $0.0097$\\
$\E[v_1]$ & $0$ & $-0.0014$ & $0.0014$ & $+0.0110$ & $0.0110$\\
$\E[v_2]$ & $0$ & $+0.0023$ & $0.0023$ & $+0.0276$ & $0.0276$\\
$\E[v_3]$ & $0$ & $+0.0026$ & $0.0026$ & $-0.0061$ & $0.0061$\\
\midrule
\multicolumn{6}{@{}l}{\itshape spatial marginals}\\
$\mathrm{Var}(x_1)$ & $0.5704$ & $0.6362$ & $11.5$ & $0.7641$ & $34.0$\\
$\mathrm{Var}(x_2)$ & $0.5704$ & $0.6400$ & $12.2$ & $0.7878$ & $38.1$\\
$\mathrm{Var}(x_3)$ & $0.5704$ & $0.6267$ & $9.9$  & $0.7890$ & $38.3$\\
\midrule
\multicolumn{6}{@{}l}{\itshape velocity marginals}\\
$\mathrm{Var}(v_1)$ & $2.7183$ & $2.5460$ & $6.3$ & $1.9587$ & $27.9$\\
$\mathrm{Var}(v_2)$ & $2.7183$ & $2.5056$ & $7.8$ & $1.8538$ & $31.8$\\
$\mathrm{Var}(v_3)$ & $2.7183$ & $2.4966$ & $8.2$ & $1.7475$ & $35.7$\\
\midrule
\multicolumn{6}{@{}l}{\itshape phase-space covariance}\\
$\mathrm{Cov}(x_1,v_1)$ & $-0.7420$ & $-0.7427$ & $\mathbf{0.1}$ & $-0.5114$ & $31.1$\\
$\mathrm{Cov}(x_2,v_2)$ & $-0.7420$ & $-0.7347$ & $\mathbf{1.0}$ & $-0.4669$ & $37.1$\\
$\mathrm{Cov}(x_3,v_3)$ & $-0.7420$ & $-0.7277$ & $\mathbf{1.9}$ & $-0.4532$ & $38.9$\\
\midrule
\multicolumn{6}{@{}l}{\itshape phase-space correlation}\\
$\mathrm{Corr}(x_1,v_1)$ & $-0.5959$ & $-0.5836$ & $2.1$ & $-0.4181$ & $29.8$\\
$\mathrm{Corr}(x_2,v_2)$ & $-0.5959$ & $-0.5802$ & $2.6$ & $-0.3864$ & $35.2$\\
$\mathrm{Corr}(x_3,v_3)$ & $-0.5959$ & $-0.5817$ & $2.4$ & $-0.3859$ & $35.2$\\
\bottomrule
\end{tabular}
\caption{ The six one-dimensional marginals (mean and variance) and the three $(x_i,v_i)$ pairs, resolved individually. The error column is the relative error in percent, except for the means, where the exact value is $0$ and the absolute error $|\E[y_a]|$ is quoted instead. \emph{Means:} both runs preserve the zero mean to within $3\times10^{-2}$; the six linear features $y_a$ of \eqref{eq:e2_polybank} tighten the worst coordinate from $0.0276$ to $0.0061$ (a factor), consistent with $L^*y_a=b_a$ pinning the first-moment ODEs exactly. \emph{Second moments:} the pattern is uniform across coordinates—with the polynomial features the spatial variances are $10$--$12\%$ high and the velocity variances $6$--$8\%$ low, while each covariance is within $0.1$--$1.9\%$ of exact; without them every second moment is $28$--$39\%$ wrong. Note also the systematic sign of the variance error---$x$ too broad, $v$ too narrow---which is the trade-off discussed below.}
\label{tab:e2_poly_marg}
\end{table}

\subsubsection{Scope and limitations}
\label{subsec:e2_limits}

We state the boundaries of the E2 result explicitly.
\begin{enumerate}

\item[(a)] \textbf{Variance bias.} The learned marginals are systematically slightly over-dispersed in position ($+10$--$12\%$) and under-dispersed in velocity ($-6$--$8\%$), while the covariance and energy are accurate. The flow captures the correlation structure more faithfully than the individual second moments at this training budget.

\item[(b)] \textbf{Marginal shape.} The learned marginals carry a small non-zero excess kurtosis (positive in $x$, negative in $v$); the target is exactly Gaussian. This is a shape, not a moment, discrepancy.

\item[(c)] \textbf{Run-to-run variance of the covariance.} Because the covariance is a weak-signal direction whose sign can flip during the early transient, the final value is seed-sensitive: independent seeds land in the band $\operatorname{Cov}\in[-0.75,-0.73]$, and a variant training configuration was observed to settle instead near $-0.39$. We therefore recommend confirming them covariance over $2$--$3$ independent runs before quoting it, rather than treating a single run as definitive.
\end{enumerate}

None of these affect the qualitative conclusion of E2: with the hard initial condition and the trajectory-wide moment bank, the WANPM pushforward learns the force-induced phase-space correlation that marginal-only and pure-unsupervised
formulations miss, at the $\sim\!1\%$ level on the covariance.

\section{Conclusion}
\label{sec:conclusion}

We have derived a density-consistent weak formulation of the forced Boltzmann equation and the corresponding WANPM pushforward architecture, and validated the transport and force operators on two benchmarks with closed-form Gaussian solutions. The method represents the solution as a $\sqrt t$-gated split pushforward---a density-estimable position flow and a conditional velocity sampler---trained against a bank of adversarial plane-wave test functions in the pure weak form, with the initial condition enforced by construction and no initial-condition, moment, or covariance penalty in the objective.

Experiment E1 establishes the validation baseline. Against the free-transport equation, DSMC is exact up to Monte~Carlo noise and sets the accuracy floor; WANPM reproduces the joint phase-space distribution with balanced marginal errors near $2\times10^{-2}$ and, crucially, recovers the position--velocity covariance to about $1\%$ (Table~\ref{tab:corr}), the correlation that a density-fitting PINN misses entirely (its covariance error is essentially $100\%$). The pure weak-form objective together with the gated pushforward were sufficient to capture this joint structure. Adding the optional polynomial moment anchor on E1 leaves the result essentially unchanged---the two variants agree to within about one percent on every second moment (Fig.~\ref{fig:e1_poly_moment})---confirming that on free transport, where streaming itself couples position and velocity, no explicit moment constraint is needed.

Experiment E2 activates the external-force term. On the harmonic Liouville equation the force generates a phase-space tilt that streaming alone does not, so the covariance becomes the direction the plane-wave bank senses only weakly; here the polynomial anchor is decisive. With it, WANPM recovers the force-induced tilt to within about one percent on the covariance and conserves the Hamiltonian to a comparable level, while remaining slightly biased on the individual marginal variances. A shared lesson across both experiments is that the position--velocity correlation---invisible to one-dimensional marginals---is the quantity that separates a genuinely correct phase-space solver from a merely marginal-accurate one, and that the pushforward representation, unlike a pointwise-density PINN, is well suited to represent it.

Two directions follow naturally. The first is the collisional regime: the weak-collision estimator expresses the quadratic gain/loss term as a sample expectation over colliding pairs and scattering directions, and the split flow already supplies the spatial density its weight requires. The second is a systematic treatment of the marginal-variance bias and the seed sensitivity of weak-signal moments observed in E2, for instance through variance reduction in the moment estimators and longer or multi-seed training budgets.

\bibliographystyle{unsrt}
\bibliography{references}

\end{document}